\documentclass{article}
\usepackage{amsmath,amssymb,amsthm}
\usepackage{mathtools}
\usepackage{xypic}
\usepackage{authblk}
\usepackage{color}
\usepackage{graphicx}
\usepackage{hyperref}
\usepackage{fancyhdr}
\hypersetup{colorlinks=true,linkcolor=black,citecolor=black}
 \usepackage[toc,page]{appendix}
\usepackage{lipsum}
\usepackage{MnSymbol}
\usepackage{algorithm}
\usepackage{multirow}
\usepackage[noend]{algpseudocode}
\usepackage{booktabs}
\usepackage{graphicx}
\usepackage{comment}
\usepackage{tikz}
\usepackage{caption}
\usepackage{graphicx} % Required for inserting images
\newcommand{\norm}[1]{\left\lVert#1\right\rVert} % Definition of the norm

\newtheorem{theorem}{Theorem}
\newtheorem{proposition}[theorem]{Proposition}

\newtheorem{corollary}[theorem]{Corollary}
\newtheorem{definition}[theorem]{Definition}
\theoremstyle{remark}
\newtheorem{example}{Example}
\newtheorem{remark}{Remark}

\def\halpha{\hat{\alpha}}
\def\hbeta{\hat{\beta}}
\def\R{\mathbb{R}}
\def\hL{\hat{L}}
\def\hvarphi{\hat{\varphi}}
\def\hphi{\hat{\phi}}
\def\e{\epsilon}
\def\a{\alpha}
\def\b{\beta}
\def\ha{\halpha}
\def\hb{\hbeta}
\def\htheta{\hat{\theta}}
\def\hx{\hat{x}}

\title{Approximating Symplectic Realizations: \newline A General Framework for the Construction of Poisson Integrators}

\author[1]{Alejandro Cabrera}
\author[2]{David Mart{\'\i}n de Diego}
\author[3]{Miguel Vaquero}
\affil[1]{UFRJ, Rio de Janeiro (alejandro@matematica.ufrj.br)}
\affil[2]{ICMAT (david.martin@icmat.es)}
\affil[3]{IE University (mvaquero@faculty.ie.edu)}
\date{}
\begin{document}
\maketitle

\begin{abstract}
While the construction of symplectic integrators for Hamiltonian dynamics is well understood, an analogous general theory for Poisson integrators is still lacking. The main challenge lies in overcoming the singular and non-linear geometric behavior of Poisson structures, such as the presence of symplectic leaves with varying dimensions. 
%One approach to circumvent this issue is to utilize (strict) symplectic realizations. 
%Symplectic realizations provide a regular symplectic structure where well-known objects, such as Lagrangian submanifolds and momentum mappings, can be employed to study Hamiltonian dynamics. 
In this paper, we propose a general approach for the construction of geometric integrators on {\it any} Poisson manifold based on independent geometric and dynamic sources of approximation. The novel geometric approximation is obtained by adapting  structural results about symplectic realizations of general Poisson manifolds.
 We also provide an error analysis for the resulting methods and illustrative applications.
%However, despite the existence of various theoretical constructions for symplectic groupoids, none of them can be directly applied to the development of numerical methods in a general way. In this paper, we propose addressing this problem by approximating the symplectic groupoid that integrates the Poisson manifold of interest. We then carefully analyze how the error of this approximation propagates in the design of Poisson integrators. Consequently, this paper aims to study geometric integrators through a novel approach that decouples the geometric and dynamic approximations.
\end{abstract}
\newpage
\tableofcontents

\section{Introduction}
Hamiltonian dynamical systems play a crucial role in many branches of science, including theoretical and applied physics, differential geometry, optimal control theory, economics, biology, robotics, and computer graphics. See~\cite{economics,marsden3,murray} and the references therein. The versatility and robustness of Hamiltonian systems make them powerful tools across diverse disciplines, enhancing our ability to model, analyze, and predict the behavior of complex systems.
%
%
%These systems, governed by Hamilton’s equations, preserve important quantities such as energy and momentum, leading to a deeper understanding of conservation laws in physics. Furthermore, the symplectic structure inherent in Hamiltonian systems ensures the stability and accuracy of numerical methods used in simulations, making them indispensable for scientific computations. This approach is strongly supported by the success seen in the symplectic setting, where the geometric integration community has demonstrated that this philosophy typically leads to more accurate integrators, both qualitatively and quantitatively. For further details about symplectic integration see~\cite{haluwa06}. Conspicuously, Hamiltonian systems are also related to the performance of optimization methods and to better simulations in statistics, see~\cite{betancourt,diakonicolas,jordan}.  Consequently, it is very desirable to develop numerical schemes that conserve the underlying geometric objects. 
%
In this context, it becomes evident the need for the development of \emph{geometric numerical integrator methods} which are adapted to the underlying geometry. 
In the case of \emph{symplectic} Hamiltonian system, such numerical methods have already been developed and applied successfully, see~\cite{haluwa06} for a general discussion.

On the other hand, it remains an open problem to extend these geometric methods to the case of general \emph{Poisson} Hamiltonian systems (\cite{Channell-Scovel,colasal23,FengQin,Ge,haluwa06,Karasozen,McLachlan_2014,LiXiWa}). In this paper, we present a general and practical approach for the construction of such Poisson integrators.

\vspace{0.25cm}

{\bf Challenges:} The main obstacles in the development of Poisson integrators arise from the singular and non-linear properties of Poisson geometry, as discussed in~\cite{Crainic-Fernades-Marcut}. Unlike symplectic geometry, where the symplectic two-form is non-degenerate, the Poisson tensor can abruptly change rank from one point to another. This phenomenon results in symplectic leaves, where solutions must lie, of varying dimensions and non-trivially glued together. Moreover, the general non-linear nature of the Poisson structure makes it impossible to rely on an explicit parametrization of the symplectic leaves, so that the method must handle them all at once. Consequently, designing numerical schemes for Poisson manifolds differs significantly from symplectic integration and requires more sophisticated techniques. Furthermore, even in the symplectic case, most methods rely on explicit Darboux coordinates, which are generally unavailable for non-canonical symplectic manifolds \cite{Albert-Kas-Ker,Austin-KWL,FengWang,haluwa06}. Thus, the particular case of designing general methods for non-canonical symplectic forms is already a challenging task.

\vspace{0.25cm}

\textbf{Setting of the problem:} We consider a general Poisson manifold $(M, \pi)$ together with a smooth Hamiltonian function $H: M \rightarrow \mathbb{R}$. Our main objective is to provide an approximation method for the flow of the corresponding Hamilton's equation,
\begin{equation}\label{eq:hameq} \dot x = \pi^\sharp(dH)|_x, \ x(t) \in M.\end{equation}
We shall denote $X_H :=\pi^\sharp(dH)\in \mathfrak{X}(M)$ the corresponding Hamiltonian vector field.  Thus, our goal is to develop a scheme that approximates the dynamics of $X_H$ while preserving the same geometric properties as the original flow:

\begin{itemize}
\item The approximation should {\it conserve the Hamiltonian} to a certain degree, corresponding to $X_H(H)=\pi(dH,dH) = 0$.
\item The approximation scheme should {\it preserve the underlying Poisson geometry}, corresponding to $\mathcal{L}_{X_H}\pi = 0$ (a known consequence of Jacobi identity for $\pi$). Additionally, we highlight the property $\mathcal{L}_{X_H}C = \pi(dH,dC) = 0$ for any Casimir function $C$. %That is, the flow of $X_H$ should preserve the Poisson tensor $\pi$ and the Casimir functions.
\end{itemize}
The condition regarding the preservation of Casimirs is recalled since it allows to have a direct control on how $x(t)$ restricts to the symplectic leaves of $M$ in concrete examples.

\medskip

{\bf The realization approach:} One interesting general approach to addressing the approximation problem highlighted above is through the introduction of an auxiliary regularizing structure: a (strict\footnote{This is the nomenclature of \cite{codawein87}, but since it will be the only type of realization considered, we do not specify it in the sequel.}) {\it symplectic realization} of $(M,\pi)$, in the sense of \cite{codawein87}. 
Such a structure comes with \emph{realization data} $R=(S,\omega,\alpha,\beta,\sigma)$ (see Sec. \ref{subsec:realizdata} below) which can be used to produce Poisson diffeomorphisms $\varphi_L:M\to M$ for each choice of so-called Lagrangian bisection $L\hookrightarrow (S,\omega)$. 
The realization approach thus consists on observing that the desired \emph{dynamic approximation} can be translated into the matter of choosing a Lagrangian bisection $\hL$ in $R$ such that $\varphi_{\hL}$ approximates the flow of our ODE \eqref{eq:hameq}. Notice that, by construction, the approximation preserves the Poisson geometry of $(M,\pi)$ exactly.

Such symplectic realizations are intimately related to (local) symplectic groupoid structures and these concepts have already been successfully exploited in references~\cite{co23,colasal23,felemamava17} to develop a theoretical framework for Poisson integrators, as well as to study the dynamics of Poisson systems \cite{vaquero2023symmetry}.
Moreover, this approach can also be seen as ``universal'' for Poisson integrators: since any Poisson diffeomorphism can be constructed using this scheme,  virtually any Poisson integrator could be designed following a general procedure outlined in~\cite{felemamava17}. 

On the other hand, to transform this theoretical approach into concrete integration methods one is strongly limited by the need to know explicitly the symplectic realization data $R$ for the given $(M,\pi)$; it is well-known that such situations are rare in practice.

\medskip

{\bf The new method:}
In this paper, we propose to combine the above realization approach with an initial \emph{geometric approximation} step which produces approximate realization data $\hat R$ for $(M,\pi)$. We implement this step by adapting known structural results about symplectic realizations. The resulting complete method, combining $\hat R$ with a dynamic approximation $\hat L$, yields a general practical integration method for every $(M,\pi)$. This methodology incorporates a decoupling between the geometric and dynamic approximations, which, to the best of the authors' knowledge, is novel and yet to be fully explored. We also provide detailed analysis of the impact of the approximation on the numerical schemes and illustrate the method in concrete examples.

In more detail, the complete method consists of two stages:
\begin{itemize}
\item \emph{Geometric Approximation, $\hat R$:} Theoretically, exact symplectic realization data $R$ always exists (\cite{ca22,cadhe26,Karasev,karasev1993nonlinear}). Nonetheless, we need an explicit description of it to incorporate into our numerical schemes. Here, we propose to approximate the symplectic realization through explicit power series expansions (building on \cite{cadhe26}). 
We show that, by constructing an approximation $\hat R$ of the geometric structures involved in the description of the realization, we indeed obtain methods that preserve the underlying Poisson geometry approximately to a certain order.

\item \emph{Dynamic Approximation, $\hL$:}
Theoretically, the Lagrangian bisection $L$ in $R$ inducing the exact flow of \eqref{eq:hameq} always exists (see a review in Section \ref{subsec:realizdata} below). In this step, we provide an approximation $\hL$ which combines with the approximate geometric data $\hat R$ into a complete approximation method for the dynamics defined by $(M,\pi,H)$ while also preserving the Poisson geometry of $(M,\pi)$ to a certain order.
\end{itemize}

\medskip

{\bf Results:}
We explore the properties of the resulting complete approximation methods which are a combination of both approximations $\hat R$ and $\hL$ above, leading to the following main contributions:

\vspace*{0.25cm}

\emph{Main Results 1:} (Geometric approximation) We present a general scheme for obtaining order $n$ (for arbitrary $n$) approximations $\hat R$ of the mapping structures defining the symplectic realization $R$ that integrates locally any Poisson manifold. We show that any Lagrangian bisection $\hL$ for $\hat R$ induces a diffeomorphism $M\to M$ which preserves the Poisson structure $\pi$ and Casimirs $C$ up to order $n$. See Section \ref{sec:geomaprox} and Theorem \ref{thm:1}.

\vspace*{0.25cm}
\emph{Main Results 2:} (Orders of the general combined method) We show that the combination of geometric and dynamic approximations, $\hat R$ and $\hL$, results in explicit numerical integrator method for \eqref{eq:hameq} and describe its total order of approximation both of the dynamics and of the underlying Poisson geometry. See Section \ref{subsec:complete} and Theorem \ref{thm:2}.

\vspace*{0.25cm}
\emph{Main Results 3:} (Implementing dynamic approximation) Building on our geometric approximation scheme $\hat R$, we provide two concrete approaches for the dynamic approximation $\hL$: one based on the Hamilton-Jacobi theory and another based on collective integrators. The theoretical properties of both approaches are carefully discussed. See Sections \ref{subsection:HJapproximation} and \ref{subsec:collective}.

\medskip

As mentioned before, these results can also be applied to symplectic Hamiltonian systems, where $\pi=\omega^{-1}$ for a non-canonical symplectic form $\omega$.
We also note that, since we aim at local numerical methods which can be applied computationally, most of the relevant constructions are specialized to the case where $M$ is replaced by a coordinate chart endowed with an arbitrary Poisson structure $\pi$. Nonetheless, we will describe most of our results and constructions in a global setting, with the local versions following from the obvious restrictions.

\bigskip

{\bf Outline of the paper:} This paper is organized as follows. In Section \ref{sec:notation}, we introduce basic concepts and notation, and we review the ingredients behind the general realization approach to Poisson integrators. In Section \ref{sec:geomaprox}, we present our construction of approximate realization data based on a truncation of Karasev's symplectic realization and study its consequences in Theorem \ref{thm:1}. In Section \ref{sec:methods}, we describe the novel class of complete methods for Poisson integrators, proving their orders of approximation in Theorem \ref{thm:2}, and describing concrete versions of them in subsections \ref{subsection:HJapproximation} and \ref{subsec:collective}. In Section \ref{sec:simulations}, we provide computational illustrations by applying the methods in concrete problems. Finally, in Section \ref{sec:conclusions}, we discuss conclusions and directions of future work.

%%%%%%%%%%%%%%%%%%%%%%%%%%%%%%%%%%%%%%%%%%%%%%%%%%%%%%
%%%%%%%%%%%%%%%%%%%%%%%%%%%%%%%%%%%%%%%%%%%%%%%%%%%%%%
%\newpage % just to make the margin comments more organized here

\section{Notation and preliminaries}\label{sec:notation}
%%%%%%%%%%%%%%%%%%%%%%%%%%%%%%%%%%%%%%%%%%%%%%%%%%%%%%
%%%%%%%%%%%%%%%%%%%%%%%%%%%%%%%%%%%%%%%%%%%%%%%%%%%%%%
In this section, we introduce the definitions and notations of the main geometric objects to be used. We also describe the \emph{general realization approach} to Poisson integrators.

%\marginAC{I removed the things about analytic functions, it was not used!}

%%%%%%%%%%%%%%%%%%%%%%%%%%%%%%%%%%%%%%%%%%%%%%%%%%%%%%%%%%%%%%%%%
\subsection{Basic Poisson and symplectic geometry}

%\paragraph{$\bullet$ Symplectic, Lagrangian, and Poisson Manifolds:}  
We recall here the basic geometric structures used along the paper. See also~\cite{AM87,marsden3} for a complete description of these topics. 

%Again, we would like also to stress that although most structures are introduced in the smooth category for the sake of completeness, we work almost exclusively in the analytic category (where all mappings are assumed to be analytic).\marginAC{Not sure...}

\medskip

    A  {\bf Poisson structure} on a  differentiable manifold $M$ is given by a bilinear map
\[
\begin{array}{rcc}
C^{\infty}(M)\times C^{\infty}(M)&\longrightarrow& C^{\infty}(M)\\
(f, g)&\longmapsto& \{f, g\}
\end{array}
\]
called the  Poisson bracket, satisfying the following properties: for all $f, g, h\in C^{\infty}(M)$,
\begin{itemize}
\item[(i)] \emph{Skew-symmetry},  $\{g, f\}=-\{f, g\}$;
\item[(ii)] \emph{Leibniz rule}, $\{fg, h\}=f\{g, h\}+g\{f, h\}$; 
\item[(iii)] \emph{Jacobi identity},  $\{\{f, g\}, h\}+\{\{h, f\}, g\}+\{\{g, h\}, f\}=0$.
\end{itemize}
We call the pair $(M, \{\,,\,\})$ a {\bf Poisson manifold}. 
A Poisson bracket is equivalent to a Poisson tensor field $\pi\in \mathfrak{X}^2(M)$ by $$\pi(df, dg)=\{f, g\}.$$ We will usually refer to a Poisson manifold as $(M,\pi)$.

\begin{example}[Dual of a Lie Algebra, $\mathfrak{g}^*$]
    If ${\mathfrak g}$ is a Lie algebra with Lie bracket $[\; ,\; ]$, then it is defined a   Poisson bracket on
 ${\mathfrak g}^*$  by 
 \[
 \{\xi, \eta \}(\alpha) = -\langle\alpha ,[\xi,\eta]\rangle\; ,
 \] where $\xi$ and $\eta \in
{\mathfrak g}$ are equivalently considered as linear forms on ${\mathfrak g}^*$, and $\alpha \in {\mathfrak g}^*$.
This  linear Poisson structure on ${\mathfrak g}^*$ is called the
{\it Kirillov-Kostant-Souriau Poisson structure}. 
\end{example}

    A {\bf Poisson map}
 between two Poisson manifolds \((M, \{ \cdot, \cdot \}_M)\) and \((N, \{ \cdot, \cdot \}_N)\) is a smooth map \(\phi: M \to N\) that preserves the Poisson brackets. This means that for any pair of smooth functions \(f, g\) on \(N\), the map \(\phi\) satisfies:
\[
\{ \phi^*f ,\phi^* g  \}_M = \phi^*\{ f, g \}_N .
\]
%In other words, the pullback of the Poisson bracket on \(N\) by \(\phi\) equals the Poisson bracket on \(M\).
%
Given a function $H\in C^\infty(M)$, its associated {\bf Hamiltonian vector field} $X_H\in \mathfrak{X}(M)$ is defined by 
\[ \{H,f\} = X_H(f), \ \forall f\in C^\infty(M).\]
The ODE on $M$ associated to the vector field $X_H$ is exactly our main equation \eqref{eq:hameq}.

\medskip

We now turn to symplectic geometry. A {\bf symplectic structure} on $M$ is a Poisson structure in which the bracket is non-degenerate,
\[ \{f,g\}= 0\, , \ \forall g\in C^\infty(M) \Rightarrow df = 0.\]
This is equivalent to the tensor $\pi_x \in \Lambda^2 T_x M$ being invertible at each $x\in M$. In this case, the inverse (as a bilinear form)
\[ \omega_x = \pi_x^{-1} \in \Lambda^2 T^*_x M\]
defines a $2$-form, $\omega \in \Omega^2(M)$, which is closed, that is $d\omega =0$. In this way, we arrive to the familiar description of a symplectic manifold as a pair $(M,\omega)$ where $M$ is a manifold and $\omega$ is a closed, non-degenerate, $2$-form $\omega$. (In this paper, the conventional signs when prescribing the inverse and Hamiltonian vector fields for $\omega$ are set so that $\{H,f\}_\omega=L_{X_H}f$.)

\begin{example}[Canonical Symplectic Form in the Cotangent Bundle]
    The cotangent bundle of a manifold, say $T^*M$, is endowed with a canonical symplectic form by taking\footnote{Notice the convention $\omega_M=dp_i \wedge dx^i$ intead of $dx^i\wedge dp_i$, which we adopt to have a direct relation to the conventions of \cite{ca22,cadhe26} in the context of symplectic realization constructions.} $\omega_M = d\theta$, where $\theta$ is the Liouville one form. In canonical coordinates $(x^i,p_i)$, these forms take the familiar expressions $\omega_M=dp_i\wedge dx^i$ and $\theta = p_i dx^i$; the corresponding Poisson brackets have $\{x^i,p_j\}_M=\delta_{ij}$ and the Hamiltonian vector field reads $X_H = -\partial_{p_i}H \ \partial_{x^i} + \partial_{x^i}H \ \partial_{p_i}$.
\end{example}

As well known, Lagrangian submanifolds play a crucial role in symplectic geometry and will also be key ingredients of this paper. A {\bf Lagrangian submanifold} $i:L\hookrightarrow M$ of a symplectic manifold $(M,\omega)$ is an embedded submanifold which is maximal among embedded isotropic submanifolds, namely, among submanifolds such that $i^*\omega =0$. We recall that $dim(M)$ is even and that $dim(L)=dim(M)/2$ for any Lagrangian submanifold.

\begin{example}[Type I Generating Functions] Given a manifold $M$ and its cotangent bundle $T^*M$, any differentiable function $S:M \rightarrow \mathbb{R}$ produces a Lagrangian submanifold in $(T^*M,\omega_M)$ by just taking $$L:=graph(dS)=\{dS|_x: \ x\in M\} \subset T^*M.$$ In canonical coordinates, $(x^i,p_i)$, this submanifold is determined by the equations $p_i=\partial_{x^i} S(x)$ with $x\in M$.
  \end{example}

%%%%%%%%%%%%%%%%%%%%%%%%%%%%%%%%%%%%%%%%%%%%%%%%%%%%%%
\subsection{Symplectic realization data, bisections and induced Poisson diffeomorphisms}\label{subsec:realizdata}
%%%%%%%%%%%%%%%%%%%%%%%%%%%%%%%%%%%%%%%%%%%%%%%%%%%%%%

The next definitions will become instrumental throughout the paper, see~\cite{codawein87,Weinstein} for more details.

\begin{definition}[Symplectic Realization]
A {\bf symplectic realization} $(S,\omega,\a,\sigma)$ of a Poisson manifold \((M, \{ \cdot, \cdot \})\) consists of a symplectic manifold \((S, \omega)\), a Poisson map \(\alpha: (S, \{ \cdot, \cdot \}_\omega) \to (M, \{ \cdot, \cdot \})\), where \(\{ \cdot, \cdot \}_\omega\) are the symplectic brackets on $S$, and a section $\sigma:M\to S$ of $\alpha$, $\alpha\circ \sigma = id_M$, such that $\sigma(M)\hookrightarrow (S,\omega)$ is a Lagrangian submanifold.
\end{definition}
As mentioned in the Introduction, this is the type of realization called ``strict'' in \cite{codawein87} and is the only type that we shall consider in this paper. 

Following that reference further, we now describe the {\bf dual map $\beta$ associated with} a given symplectic realization $(S,\omega,\alpha,\sigma)$.
This map $\beta:U\subset S \to M$ is a submersion defined on a neighborhood $U$ of $\sigma(M)$ and is characterized by the properties:
\begin{itemize}
    \item $\sigma$ is also a section of $\beta$, $\beta\circ \sigma = id_M$,
    \item the $\beta$-fibers and the $\alpha$-fibers are symplectic orthogonal,
    \[ Ker(D_y\beta)^\omega = Ker(D_y\alpha), \ y\in U. \]
\end{itemize}
In \cite{codawein87}, it is shown that, given a symplectic realization, such a dual map $\beta$ exists, that it is an anti-Poisson morphism (that is, a Poisson map onto $(M,-\pi)$), and that its germ around $\sigma(M)$ is uniquely determined by the realization. 

By replacing $S$ with a suitable small enough neighborhood of $\sigma(M)$, we can assume that both maps are defined on the entire $S$. Our constructions will indeed be only based on local considerations around $\sigma(M)\subset S$. Moreover, independently of how the maps are obtained, our constructions will only depend on having such maps $\alpha,\beta,\sigma$ with the properties highlighted above. We resume this discussion with the following definition.

\begin{definition}
    We say that  $R:=(S,\omega,\alpha,\beta,\sigma)$ defines {\bf realization data} for $(M,\pi)$ when:
    \begin{itemize}
        \item $(S,\omega)$ is a symplectic manifold,
        \item $\alpha,\beta:(S,\omega)\to (M,\pi)$ are Poisson and anti-Poisson maps, respectivelly, having $\omega$-orthogonal fibers,
        \item $\sigma$ is a section for both $\alpha$ and $\beta$, and $\sigma(M)\hookrightarrow (S,\omega)$ is Lagrangian.
    \end{itemize}

\end{definition}

\begin{remark}[Dual pairs]
    The data given by the two maps $\alpha,\beta:(S,\omega)\to (M,\pi)$, with $\alpha$ being Poisson and $\beta$ being anti-Poisson, with domain a symplectic manifold and with symplectically orthogonal fibers, is known as a \emph{dual pair}. All the constructions in this paper can be formulated as originating from such a dual pair with the extra data of a map $\sigma:M\to S$ which is a section for both $\alpha$ and $\beta$, and whose image is Lagrangian in $S$. Such a section for both $\alpha$ and $\beta$ determines $\sigma(M)\subset S$ a type of submanifold called \emph{bisection} which will play a key role below.
\end{remark}

\begin{remark}[Symplectic groupoids]\label{rmk:groupoids}
    Another important result from \cite{codawein87} is that a symplectic realization $(S,\omega,\alpha,\sigma)$ as above determines a unique germ of local symplectic groupoid structure $(S,\omega)\rightrightarrows M$ around $\sigma(M)$. In this local groupoid, $S$ is the set of arrows, $M$ is the set of objects, $\alpha$ is the source map, and $\sigma$ is the identities map. Moreover, the dual map $\beta$ is the target map, there is an inversion map $inv:U_{inv}\subset S \to S$ defined around $\sigma(M)$, and there is a multiplication map 
    \[ m: U_{(2)} \subset \{ (z_1,z_2)\in S\times S: \alpha(z_1)=\beta(z_2)\} \to S\]
    defined on a neighborhood of $\{(\sigma(x),\sigma(x)):x\in M\}$. All the structure maps satisfy local versions of the algebraic groupoid axioms around the units and $\omega$ becomes \emph{multiplicative}. Conversely, any local (or global) symplectic groupoid determines a symplectic realization with $(S,\omega)$ the space of arrows, $\alpha$ the source map, and $\sigma$ the identity map.
    
    We thus see that (local or global) symplectic groupoids provide examples of the symplectic realization data needed for our method. On the other hand, we emphasize that \emph{the multiplication map does not enter the construction of this paper}. (The incorporation of the multiplication map into the methods will be explored elsewhere, \cite{cacanar}.)
\end{remark}

\medskip

We now describe how to use realization data for $(M,\pi)$ in order to produce Poisson diffeomorphisms. The idea is that this approach can be used to provide Poisson maps which approximate the flow of our ODE \eqref{eq:hameq}.
%
%To this end, we introduce Lagrangian bisections for the realization data and explain how they induce Poisson maps on the underlying $(M,\pi)$.

\begin{definition}[Lagrangian Bisection and the induced mapping]\label{def:lagrangian_bisection} Given realization data $(S,\omega,\alpha,\beta,\sigma)$ for $(M,\pi)$, a {\bf  bisection} is a submanifold $L \subset S$ such that both restrictions $\alpha|_L,\beta|_L: L \to M$ are diffeormofisms. When, additionally, $L\hookrightarrow (S,\omega)$ is Lagrangian, we say that it determines a {\bf Lagrangian bisection}.
   A bisection $L$ determines an {\bf induced diffeormorphism} by the rule
    \[ \varphi_L=\alpha\circ(\beta_{|L})^{-1}:M\rightarrow M.\]
\end{definition}
Note that, replacing $M$ with a suitable open, one obtains an analogous notion of {\bf local} (Lagrangian) bisection and its induced locally defined diffeomorphism of $M$. Also note that $\sigma(M)$ is, by definition, a Lagrangian bisection and that the induced map is the identity,
\[ \varphi_{\sigma(M)}=id_M. \]

The following compilation of results is the heart of the \emph{realization approach} for Poisson integrators.%\marginAC{XX find exact reference!}
\begin{theorem}\label{CDW}(\cite{codawein87}) Let $R=(S,\omega,\alpha,\beta,\sigma)$ be realization data for $(M,\pi)$. Then, 
\begin{enumerate}
%\item if $\alpha|_L$ is a diffeomorphism $\Rightarrow$ $L$ is a Lagrangian bisection, \marginAC{need $L$ small ?}
\item when $L\hookrightarrow (S,\omega)$ is a Lagrangian bisection for $R$, the induced map $\varphi_L$ defines a Poisson diffeomorphism $(M,\pi)\to (M,\pi)$.
\item if $\phi^H_t:M \to M$ is the Hamiltonian flow on $(M,\pi)$ defined by $H\in C^\infty(M)$, then
\[ \phi^H_t = \a \circ \phi^{\a^*H}_t \circ \sigma,\]
with $\phi^{\a^*H}_t:S\to S$ the Hamiltonian flow of $\alpha^*H$ in $(S,\omega)$.
\item  In the setting of the previous item,
\[ \phi_t^H = \varphi_{L_t}\text{ for the Lagrangian bisection }L_t = \phi^{\alpha^*H}_t(\sigma(M))\, .\]
 
\end{enumerate}
\end{theorem}
We observe that, since $\a$ and $\b$ have symplectically orthogonal fibers, it follows that 
\[ \text{$\b$ is conserved along the flow $\phi^{\a^*H}_t$ in $(S,\omega)$.}\]
This fact explains the last item.
We also observe that, while the former are Poisson objects (singular), the latter are symplectic constructions (regular); this is at the heart of the use of symplectic realizations in numerical methods for $(M,\pi,H)$.

%%%%%%%%%%%%%%%%%%%%%%%%%%%%%%%%%
\subsection{The general realization approach to Poisson integrators}\label{subsec:realmeth}
 We can now describe in detail the general realization approach to approximately integrate the flow of \eqref{eq:hameq} on $(M,\pi,H)$, see ~\cite{co23,colasal23,felemamava17}. The key assumption is that we have at hand realization data $R=(S,\omega,\alpha,\beta,\sigma)$ for $(M,\pi)$.

From Thm. \ref{CDW}, we know that the exact flow $\phi^H_t:M \to M$ of eq. \eqref{eq:hameq}, for $t$ small enough, can be presented as an induced map $\varphi_{L_t}$ for the Lagrangian bisection $L_t = \phi^{\alpha^*H}_t(\sigma(M))$ in $R$. The key idea is then to replace $L_t$ by an approximation $\hL_t$, which is still a Lagrangian bisection for $R$, and then define the approximate flow $\hat \phi^H_t:M \to M$ of our ODE \eqref{eq:hameq} as
\[ \hat \phi^H_t:= \varphi_{\hL_t} = \alpha\circ \beta|_{\hL_t}^{-1}: M \to M.\]
The upshot of the method is that, by construction, $\hat \phi^H_t$ is a family of Poisson diffeomorphisms in $(M,\pi)$, thus preserving the ambient Poisson geometry. The degree of approximation of $\hat \phi^H_t$ to $\phi^H_t$ will depend on how well $\hL_t$ approximates $L_t$, see a discussion in \cite{co23}.

\begin{remark}[Rescaling $\pi$]\label{rmk:rescaling}
    Let us denote $\Phi^{\pi,H}_t:M\to M$ the Hamiltonian flow associated to $(M,\pi,H)$, putting in relevance in the notation the dependence on the Poisson tensor $\pi$. From the structure of the underlying ODE \eqref{eq:hameq}, it follows directly that rescaling time is equivalent to rescaling $\pi$ or $H$,
    \[ \Phi_{\lambda t}^{\pi,H}=\Phi_{t}^{\lambda\pi,H} = \Phi_{t}^{\pi,\lambda H}. \]
\end{remark}
 Having the above Remark in mind, let us assume that $R_\e=(S,\omega,\alpha_\epsilon,\beta_\epsilon,\sigma)$ is a family of realization data for the $\epsilon$-family $(M,\epsilon \pi), \ \epsilon\geq 0$. Fixing an additional ``time ste'' parameter $h>0$, we can then conclude that
    \begin{equation} \label{eq:rescaledflow}
    \phi_{t=\epsilon h}^H = \Phi^{\pi,H}_{\epsilon h} =\Phi^{\epsilon\pi,H}_{h} = \alpha_\epsilon\circ \beta_\epsilon|_{L_{\e,h}}^{-1} \text{ with }L_{\e,h}=\phi^{\a_\e^*H}_h(\sigma(M)).
    \end{equation}
    This is the precise approach we shall take in the sequel, in which we shall introduce both: \emph{geometric approximations} of order $O(\epsilon^{n+1})$ for the data $\alpha_\epsilon,\beta_\epsilon$ and \emph{dynamic approximations} $\hL_{\e,h}$ for $L_{\e,h}$ and each $\e$, as before. (See Section \ref{subsec:complete} where this discussion resumes.)

%%%%%%%%%%%%%%%%%%%%%%%%%%%%%%%%%%%%%%%%%%%%%%%%%%%%%%
%%%%%%%%%%%%%%%%%%%%%%%%%%%%%%%%%%%%%%%%%%%%%%%%%%%%%%
\section{Geometric approximation: the realization data}\label{sec:geomaprox}
%%%%%%%%%%%%%%%%%%%%%%%%%%%%%%%%%%%%%%%%%%%%%%%%%%%%%%
%%%%%%%%%%%%%%%%%%%%%%%%%%%%%%%%%%%%%%%%%%%%%%%%%%%%%%
In this section, we first review the construction of the Karasev symplectic realization (\cite{ca22,cadhe26}) and then present an approximation scheme for this symplectic realization. 
%in the context of analytic Poisson structures. 
The main objective is to demonstrate that if the realization data is approximated to order $n$, then the diffeomorphism $\varphi_L$ induced by a Lagrangian bisection $L$ is a Poisson diffeomorphisms up to order $n$. This will ensure that the error of the induced Poisson mappings can be controlled to a desired accuracy.

As explained in the Introduction, with an eye on computer-implementable methods, we mostly restrict to the case in which $M\simeq \mathbb{R}^n$ endowed with an arbitrary Poisson structure $\pi$.
%

%%%%%%%%%%%%%%%%%%%%%%%%%%%%%%%%%%%%%%%%%%%%%%%%%%%%%%
\subsection{Review of the Karasev realization}\label{subsec:symplectic_groupoid_construction}
%%%%%%%%%%%%%%%%%%%%%%%%%%%%%%%%%%%%%%%%%%%%%%%%%%%%%%

There are two local symplectic realizations that have been extensively studied in the literature: the Weinstein symplectic realization (\cite{Weinstein}) and the Karasev approach (\cite{karasev1993nonlinear}). In both cases, the realization space is given by $S\subset T^*M$ a neighborhood of the zero section. In the Weinstein case, the symplectic structure $\omega$ is a deformation of the canonical symplectic structure $\omega_M$ while the realization map is the projection $S\subset T^*M\to M$. In the Karasev case, the symplectic structure is canonical $\omega_M$ and the realization map $\alpha:S\to M$ is a deformation of the projection.
The connection between the two can be seen in~\cite{cadhe26}. 

%{\tiny Although not very relevant to this paper, it is worth mentioning that global obstructions to the integrability of Poisson manifolds have also been found; see~\cite{Crainic2002IntegrabilityOP}.}

In this paper, we shall work with the {\bf Karasev realization} which we now review. Let \(M=\mathbb{R}^n\) with coordinates $x^i$ and consider a general Poisson tensor \(\pi = \pi^{ij}(x) \partial x_i \wedge \partial x_j\). We want to construct realization data $(S,\omega,\alpha_\epsilon,\beta_\epsilon,\sigma)$ for $(M,\epsilon \pi)$, for any $\epsilon \geq 0$. 

To this end, we follow the presentation of \cite{cadhe26} and  introduce the following auxiliary objects. On \(T^*M = \mathbb{R}^n \times \mathbb{R}^n\), we consider canonical coordinates \((x^i, p_j)\) and the ``flat Poisson spray'' vector field
\begin{equation}
     \overline{V}(x,p) = -\pi(x)^{ij}p_i\partial x^j \in T_{(x,p)}(T^*M). \tag{$\overline{V}$}
    \end{equation}
Notice that the \(p\)-variables do not evolve. We denote the corresponding flow at time \(t\) by \(\varphi^{\overline{V}}_t\).
Next, we consider the corresponding ``$x$-average'' mapping \(\phi_\epsilon:T^*M \to M\) is defined by  
    \begin{equation}
    \phi^i_\epsilon(x,p) = \frac{1}{\epsilon}\int_0^\epsilon \left((\varphi^{\overline{V}}_s)^*x^i\right)\big|_{(x,p)} \, ds. \tag{$\phi_\epsilon$}
    \end{equation}

Then, following Karasev \cite{Karasev}, we define $\alpha_\e: U_\e \subset T^*M \to M$ on a neighborhood of $p=0$, via the implicit relation
         \begin{equation}\label{eq:alphaKar}
\phi_\epsilon(\alpha_\epsilon(x,p),p) = x.  
         \end{equation}
By the implicit function theorem, it follows that the above formula indeed defines a smooth map on a neighborhood $U_\e$ of the zero section $$0_M:=\{(x,p)\in T^*M:p=0\}\subset T^*M.$$
Note that $U_\e$ grows as $\e \to 0$ and that
\[ \alpha_\e(x,0)=\alpha_0(x,p) = x. \]

We thus obtain the following realization data.
\begin{definition}
    Let $(M\simeq \R^n, \pi)$ be Poisson and $\alpha_\e$ the map defined by \eqref{eq:alphaKar}. Then,
    \[ K_\e := (S,\omega,\alpha_\e,\beta_\e,\sigma) \]
    defines the {\bf Karasev realization data} for $(M,\pi)$, where $S\subset T^*M$ is a suitable small enough neighborhood of $0_M \subset T^*M$, $\omega=\omega_M$ is the canonical symplectic structure, 
    \[\beta_\e (x,p) = \alpha_\e(x,-p)\]
    and $\sigma(x)=(x,0)$.
\end{definition}
Details about why these maps indeed define realization data for $(M,\pi)$ can be found in \cite[\S 3.3]{ca22}.

\begin{remark}[Associated local Symplectic Groupoid]\label{rmk:Klocgd} Following \cite{ca22} further, the realization data $K_\e$ can be enriched to a local symplectic groupoid structure $G_\e \rightrightarrows M$ for each $\e\geq 0$. For latter use, in this local groupoid $G_\e$, the inversion map is given by $inv(x,p)=(x,-p)$. As remarked before, the multiplication map will not be used in this paper.
\end{remark}

\begin{remark}[Rescaling properties of $\alpha_\e$]\label{rmk:alpharescaling}
    We recall from \cite[Lemma 3.22]{ca22} that the Karasev realization $\alpha_\e$ has special rescaling properties:
    \[ \alpha_{\lambda \e}(x,p)=\alpha_\e(x,\lambda p).\]
    It then follows that
    \[ \phi^{\alpha_\e^*H}_{\lambda t} \circ \mu_\lambda= \mu_{\lambda}\circ\phi^{\alpha_{\lambda \e}^*H}_t ,\]
    where $\mu_\lambda(x,p)=(x,\lambda p)$ and $\lambda \in \R$.
\end{remark}

%%%%%%%%%%%%%%%%%%%%%%%%%%%%%%%%%%%%%%%%%%%%%%%%%%%%%%
\subsection{Approximate realization data through truncation}\label{subsec:Ktruncated}
%%%%%%%%%%%%%%%%%%%%%%%%%%%%%%%%%%%%%%%%%%%%%%%%%%%%%%
The idea in this subsection is to produce order $n$ approximations for the maps $\alpha_\e$ and $\beta_\e$ in the Karasev realization data $K_\e$. Before doing that, we define more generally such approximations.

To this end, we first recall that, for smooth maps $f,\hat f:(D\subset \R^n) \times I \to \R^m$ with $I\subset \R$ an interval containing $0$ and $D$ an open domain, we say that $\hat{f}$ is an  {\bf approximation of order $n$ of} $f$, denoted $f = \hat f \ mod \ \mathcal{O}(\e^{n+1})$, when 
\[
\norm{f(x,\epsilon) - \hat{f}(x,\epsilon)} = \mathcal{O}(\epsilon^{n + 1}),
\]
uniformly for $x$ varying in any compact in $D$, and for $\norm{\cdot{} }$  any norm in $\mathbb{R}^{m}$. 
We also recall the notation $F(\e) = \mathcal{O}(\e^k)$ if $|F(\e)|\leq C \e^k$ for some $C>0$ and for all $\e$ in a small enough neighborhood of $\e=0$. Notice that the definition of $f = \hat f \ mod \ \mathcal{O}(\e^{n+1})$ can be extended to $\e$-families defined on smooth manifolds $f,\hat f: M \times I \to N$. Nevertheless, as mentioned in the Introduction, we will be mostly interested on the cases $M\simeq \R^n, N\simeq \R^m$, so we do not review the details of the general case.

\begin{definition}
    Let $R_\e = (S,\omega,\alpha_\e,\beta_\e,\sigma)$ be realization data for $(M,\e \pi), \ \e\geq 0$. We say that $\hat R_\e=(S,\omega,\halpha_\e,\hbeta_\e,\sigma)$ defines {\bf approximate realization data of order $n$} when the maps $\halpha_\e$ and $\hbeta_\e$ are approximations of order $n$ of $\alpha_\e$ and $\beta_\e$, respectively, and with the additional condition that $\sigma$ is a bisection for all $\e$: $\halpha_\e \circ \sigma = \hbeta_\e \circ \sigma = id_M$.
\end{definition}
Let us note that, for any submanifold $L\subset S$ which is a (possibly local) bisection for both $\halpha_\e$ and $\hbeta_\e$ in the obvious sense, the corresponding induced map
\[ \hvarphi_L = \halpha_\e \circ \hbeta_\e|_L^{-1}: M \to M \]
is a well defined (possibly local) diffeomorphism.

\medskip

Next, we observe that, clearly, order $n$ approximations can be obtained for a smooth $\e$-family of functions $f_\e(x)=f(x,\e)$ by {\bf truncating its Taylor expansion} around $\e=0$,
\[ \tau^n[f_\e](x) := \sum_{k=0}^n \frac{\e^k}{k!} \partial^k_\e f|_{(x,0)}. \]

We thus come back to the Karasev realization data $K_\e = (S\subset T^*M,\omega_M, \alpha_\e,\beta_\e,\sigma)$ of $(M=\R^n,\e \pi)$ and consider the corresponding approximate realization data of order $n$,
\[ \hat K_\e = (S\subset T^*M,\omega_M, \halpha_\e:=\tau^n[\alpha_\e],\hbeta_\e:=\tau^n[\beta_\e],\sigma).\]
We call this the {\bf approximate Karasev realization data} of order $n$.
Note that the condition that $\sigma(x)=(x,0)$ is a bisection for $\halpha_\e,\hbeta_\e$ follows directly from
the rescaling property in Remark \ref{rmk:alpharescaling}.

\medskip

Finally, we observe that the Taylor expansion of $\alpha_\e$ around $\e=0$ (and, hence, of $\beta_\e$) in the Karasev realization was explicitly computed in \cite{cadhe26}, by means of the defining relation \eqref{eq:alphaKar}. In that reference, the coefficients in the expansion
\[ \tau^n[\alpha^i_\e](x,p) = \sum_{k=0}^n  \alpha_{(k)}^i(x,p) \ \e^k \]
where explicitly given in alternative ways: by an explicit recursion \cite[eq. (3.5)]{cadhe26}; by an explicit formula in terms of rooted trees and elementary differentials of the vector field $\overline{V}$ in \cite[eq. (3.13)]{cadhe26} with weights defined by \cite[eq. (3.12)]{cadhe26} or by the iterated integrals \cite[eq. (3.15)]{cadhe26}. 

We shall not be needing the precise formulas here, but is important to have in mind that these explicit formulas for $\tau^n[\alpha_\e](x,p)$ and $\tau^n[\beta_\e](x,p)=\tau^n[\alpha_\e](x,-p)$, which can be found in the references above, are such that \emph{our overall method can be indeed implemented practically in a computer}. For completeness, we write the first terms of the general expansion \cite{cadhe26}:
\[ \alpha_\e^{i}(x,p)=x^{i} + \frac{\epsilon}{2}\pi^{vi}p_{v}+\frac{\epsilon^{2}}{12}\partial_{u}\pi^{vi}\pi^{wu}p_{v}p_{w} + \frac{\epsilon^{3}}{48}\partial_{u}\partial_{w}\pi^{vi}\pi^{ku}\pi^{lw}p_{v}p_{k}p_{w}+\mathcal{O}(\epsilon^{4}).
\]

\begin{remark}[Rescaling properties of $\hat K_\e$]\label{rmk:rescalingKaprox}
    The truncated $\ha_\e$ has the same rescaling property as $\a_\e$ recalled in Remark \ref{rmk:alpharescaling}. It analogously follows that
    \[ \phi^{\ha_\e^*H}_{\lambda t} \circ \mu_\lambda = \mu_\lambda \circ \phi^{\ha_{\lambda \e}^*H}_t,\]
    for any Hamiltonian $H\in C^\infty(M)$.
\end{remark}

\begin{remark}[Relation to Kontsevich's quantization]
    In \cite{cadhe26}, it was further shown that the Taylor expansion of Karasev's realization $\alpha_\e$ is directly related to (the tree level part of) Kontsevich's quantization formula for $(M=\R^n,\pi)$. This thus establishes a non-trivial connection between that formula and the present methods for Poisson integrators.
\end{remark}

\begin{remark}[Other approximations of $K_\e$] We observe that other approximation approaches can be taken. In particular, one can implement a computational method to approximate the Karasev map $\a_\e$ by solving for the defining relation \eqref{eq:alphaKar} in an approximate way (e.g. replacing the integral by an approximating sum and the flow $\varphi^{\overline{V}}_s$ by a numerical approximation). We shall explore these practical possibilities elsewhere.
\end{remark}

%%%%%%%%%%%%%%%%%%%%%%%%%%%%%%%%%%%%%%%%%%%%%%%%%%%%%%%%%%%%%%%%%%%%%%%%%
\subsection{Order of geometry preservation under approximate realizations}
%%%%%%%%%%%%%%%%%%%%%%%%%%%%%%%%%%%%%%%%%%%%%%%%%%%%%%%%%%%%%%%%%%%%%%%%%
We now explore the degree of approximate preservation of $\pi$ for a map $\hvarphi_{\e,L}$ induced by a Lagrangian bisection $L\subset S$ with respect to approximations $\halpha_\e$ and $\hbeta_\e$ of the realization data.

\medskip

First, observe that when $L\hookrightarrow (S,\omega)$ is a Lagrangian submanifold which is close enough to the Lagrangian $\sigma(M) \subset S$, then $L$ is a bisection relative to any approximate realization data $\hat R_\e$. This follows by transversality since, by definition, $\sigma(M)$ defines a bisection for the underlying maps $\halpha_\e, \hbeta_\e$ for all $\e$. Similarly, $L$ is also a bisection for the exact $R_\e$ when is close enough to $\sigma(M)$.

\medskip

%%%%%%%%%%%%%% THEOREM
\begin{theorem}[First main result]\label{thm:1}\label{main_theorem} Let $(M,\pi)$ be Poisson and $\hat R_\e = (S,\omega,\halpha_\e,\hbeta_\e,\sigma)$ be approximate realization data of order $n$ for $(M,\e\pi)$. Consider $L\hookrightarrow (S,\omega)$ a Lagrangian which is close enough to $\sigma(M)$.
Then, $L$ is a bisection for $\halpha_\e$ and $\hbeta_\e$, and the induced map
\[ \hvarphi_{\e,L} = \halpha_\e \circ \hbeta_\e|_L^{-1}:M \to M  \]
preserves the Poisson tensor $\pi$ and any Casimir $C\in C^\infty(M)$ up to order $n$:
\begin{equation}
\hvarphi_{\e,L}^*\pi - \pi= \mathcal{O}(\epsilon^{n + 1}), \ \ \hvarphi_{\e,L}^*C - C = \mathcal{O}(\epsilon^{n + 1}),
\end{equation}
uniformly on any compact in $M$.
\end{theorem}

\begin{proof}
It is clear that we can reduce the proof to the case $M=\R^n$ and, by the Lagrangian tubular neighborhood theorem applied to $\sigma(M)\hookrightarrow (S,\omega)$, to $S\simeq \R^{2n}$.

    We thus begin the proof with the following general facts: 
    % %
    % \marginDM{Quizás añadir que $\epsilon\in I\in {\mathbb R}$ with $I$ an opne interval with $0\in I$.

    % Recordad que el error de Marsden y West fue suponer alguna de estas propiedades para una función con problemas exactamente en el 0.
    % Se puede citar \cite{Cuell-Patrick}
    % }
    % %
    % \marginMV{La función ya está definida en un entorno del $0$, es solo que nos fijamos en un trozo $[0,\epsilon_0)$.}
    \begin{enumerate}
        \item if $f_\e:\R^n\to \R^n$ is a smooth family of diffeomorphisms and $\hat f_\e$ is a family of diffeomorphism which is an order $n$ approximation of $f_\e$, then $\hat f_\e^{-1}$ is an order $n$ approximation of $f_\e^{-1}$;

        \item if $f_\e,g_\e$ are composable $\e$-families of maps and $\hat f_\e,\hat g_\e$ are corresponding order $n$ approximations, then $\hat f_\e \circ \hat g_\e$ is an order $n$ approximation of $f_\e\circ g_\e$.
    \end{enumerate}
    %
% \marginMV{Creo que está implícito, pero para que esto funcióne las derivadas de $f$ tienen que ser distintas de $0$. Si no la inversa no es derivable/analítica.}
    %
The first fact (1) follows by observing that the $n$-first Taylor coefficients of the inverse $f^{-1}_\e$ of any family $f_\e$ are completely determined the $n$-first coefficients of $f_\e$ by a recursion following from expanding $f_\e \circ f_\e^{-1} = id$. Then, since $f_\e$ and $\hat f_\e$ share the same $n$-first coefficients by hypothesis, the result follows. Similarly, the second fact (2) follows from the fact that the $n$-first Taylor coefficients of the composition $f_\e \circ g_\e$ are completely determined by the $n$-first coefficients of $f_\e$ and $g_\e$.

Coming back to the proof of the theorem, let $R_\e=(S,\omega,\alpha_\e,\beta_\e,\sigma)$ be the realization data of which $\hat R_\e$ is an order $n$ approximation. Since $L$ is close to $L_0:=\sigma(M)$ and $L_0$ is a bisection for $\alpha_\e$ and $\b_\e$, then $L$ is also a bisection for $\a_\e,\b_\e$. %\marginAC{X check if this is the relevant argument and not $\e\to 0$} 
From the general results (see Thm. \ref{CDW}), since $L\hookrightarrow (S,\omega)$ is Lagrangian, we thus know that \[(a):\text{$\varphi_{\e,L}=\a_\e \circ \b_\e|_L^{-1}$ is a Poisson diffeomorphism on $(M,\pi)$. }\]
Note that, a priori, it is a Poisson map for $(M,\e \pi)$ but this implies that it must preserve $\pi$ by linearity of $\varphi_{\e,L}^*$ and for $\e\neq 0$.

Next, we use (1.) above to conclude that $\hbeta_\e|_L^{-1}$ is an order $n$ approximation of $\b_\e|_L^{-1}$ and combine this with (2.) to obtain:
\[(b): \hvarphi_{\e,L}\text{ is an order $n$ approximation of }\varphi_{\e,L}. \]

Finally, we have:
\[ \mathcal{O}(\e^{n+1}) \overset{(b)}{=} \hvarphi_{\e,L}^*\pi - \varphi^*_{\e,L}\pi \overset{(a)}{=} \hvarphi_{\e,L}^*\pi - \pi,  \] 
and analogous for Casimirs $C$. This finishes the proof.
\end{proof}

%%%%%%%%%%%%%%%%%%%%%%%%%%%%%%%%%%%%%%%%%%%%%%%%%%%%%%
%%%%%%%%%%%%%%%%%%%%%%%%%%%%%%%%%%%%%%%%%%%%%%%%%%%%%%
\section{Complete methods and dynamic approximations}\label{sec:methods}
%%%%%%%%%%%%%%%%%%%%%%%%%%%%%%%%%%%%%%%%%%%%%%%%%%%%%%
%%%%%%%%%%%%%%%%%%%%%%%%%%%%%%%%%%%%%%%%%%%%%%%%%%%%%%

In this section, we first describe the general complete method proposed in this paper and prove an estimate on its order of approximation, with respect to both its geometric and dynamic properties. We then present two different concrete strategies for dynamic approximation and describe the resulting complete Poisson integrator methods. The first is based on Hamilton-Jacobi theory and the second is based on approximating the Hamiltonian flow of $\hat \a_\e^*H: S\rightarrow {\mathbb R}$ on the symplectic realization space (collective integrators).

%%%%%%%%%%%%%%%%%%%%%%%%%%%%%%%%5

\subsection{The general complete method and its orders of approximation}\label{subsec:complete}

Let us recall the outline of our construction of a method of approximation of the flow $\phi^H_t$ of eq. \eqref{eq:hameq} associated to a Poisson Hamiltonian system $(M,\pi,H)$. As explained in Section \ref{subsec:realmeth}, we shall focus on approximating $\phi^H_{t=\epsilon h}$, for independent parameters $\epsilon,h \sim 0$, following eq. \eqref{eq:rescaledflow}. 

\medskip

Given $(M,\pi)$ Poisson, since ``exact'' symplectic realization data $R_\e$ is hard to obtain explicitly, we work with a (general) \emph{approximate realization data} as described in Section \ref{sec:geomaprox},
$$\hat R_\epsilon= (S, \omega,\hat \alpha_\e, \hat \beta_\e, \sigma).$$
To complete the method, following the general \emph{realization approach} recalled in Sec. \ref{subsec:realmeth}, we need to provide the \emph{dynamic approximation} ingredients. Namely, given $H\in C^\infty(M)$, having in mind the rescaling of eq. \eqref{eq:rescaledflow}, we need to construct a $2$-parameter family of Lagrangians $\hL_{\e,h} \hookrightarrow (S,\omega)$ such that: for $\e, h\sim 0$ small enough,
\begin{enumerate}
    \item $\hL_{\e,h}$ is a bisection for $\halpha_\epsilon,\hbeta_\epsilon$, 
    \item the induced map 
    \[\hphi_{\epsilon,h}:=\halpha_\epsilon \circ \hbeta_\epsilon|_{\hL_{\e,h}}^{-1}:M\to M\]
    both preserves approximately $\pi$ and approximates the exact flow $\phi^H_{t=\epsilon h}$, up to desired orders.
\end{enumerate}
Note that $\epsilon$ controls the geometric approximation order while the pair $(\e,h)$ controls the dynamic approximation order. We will restrict ourselves to the case in which
\[ \hL_{\e,0} = \sigma(M) \text{ corresponding to } \phi^H_{t=0}=id_M. \]
To analyze the dynamic approximation order of $\hphi_{\e,h}$, we introduce the following notion of degree of approximation of such an $\hL_{\e,h}$. Let $R_\e=(S,\omega,\a_\e,\b_\e,\sigma)$ be the ``exact'' realization data underlying $\hat R_\e$ and
$$L_{\e, h}=\phi^{\alpha_\epsilon^*H}_h(\sigma(M)) \hookrightarrow (S,\omega)$$
the ``exact'' dynamic data inducing our dynamics, $\varphi_{L_{\e, h}}=\phi^H_{t=\e h}$, with respect to $R_\e$ (see Thm. \ref{CDW} and Rmk. \ref{rmk:rescaling}).
To compare $\hL_{\e,h}$ and $L_{\e, h}$, we first notice that, by the Lagrangian tubular neighborhood theorem applied to $\sigma(M)\hookrightarrow (S,\omega)$, we can assume $S\subset T^*M$ is a neighborhood of the zero section $0_M$, that $\sigma(M)=0_M$, and that $\omega=\omega_M$ the canonical symplectic structure. We can further assume that any Lagrangian  $L\subset T^*M$ sufficiently close to $\sigma(M)$ is horizontal: there exists a closed $1$-form $\theta_L \in \Omega^1(M)$ such that $L=\theta_L(M)$ \cite{We71}. 

In this setting, we say that a Lagrangian bisection $\hL_{\e,h}$ for $\hat R_\e$, with $\hL_{\e,0}=\sigma(M)$ as above, {\bf approximates $L_{\e, h}$ to order $m$ in $h$} when the corresponding family of $1$-forms $\theta_{\hL_{\e,h}}$ approximates $\theta_{L_{\e, h}}$ as $1$-forms to order $\mathcal{O}(h^{m+1})$, uniformly on compacts in $M$ and uniformly for $\e \sim 0$.

\medskip

The following result computes the order of approximation of the complete method.
\begin{theorem}[Second main result]\label{thm:2}
    Let $(M,\pi)$ be Poisson and $H\in C^\infty(M)$ with associated Hamiltonian flow $\phi^H_{t}:M\to M$ for the ODE \eqref{eq:hameq}.
    Consider a realization data $R_\epsilon=(S,\omega,\alpha_\epsilon,\beta_\epsilon,\sigma)$ for $(M,\epsilon \pi), \e \geq 0$, and $L_{\e, h}=\phi^{\alpha_\epsilon^*H}_h(\sigma(M))$ the associated Lagrangian bisection inducing $\phi^H_{t=\e h}$, for $h\geq 0$, as recalled above.

    Assume $\halpha_\e,\hbeta_\e$ are approximations of $\alpha_\e,\beta_\e$ of order $n$ in $\e$ and, for each $\e$, $\hL_{\e,h} \hookrightarrow (S,\omega)$ is a Lagrangian submanifold approximating  $L_{\e, h}$ to order $m$ in $h$, with $\hL_{\e,0}=\sigma(M)$. Then, for $\e,h\sim 0$ small enough, $\hL_{\e,h}$ is a bisection for $\halpha_\e,\hbeta_\e$ and the induced map
    \[ \hphi_{\e,h}= \halpha_\epsilon \circ \hbeta_\epsilon|_{\hL_{\e,h}}^{-1}:M\to M\]
    satifies the following properties: it preserves $\pi$ and Casimirs $C\in C^\infty(M)$ up to order $n$ in $\e$; it defines an approximation of the dynamics $\phi^H_{t=\e h}$ up to $\mathcal{O}(\e^{n+1}) + \mathcal{O}(h^{m+1})$ error terms. In particular, the combined order of dynamic approximation is $\min\{ n, m\}$.
\end{theorem}

\begin{proof}
We first observe that, as before, we can reduce the proof to $M=\R^n$ and $S\subset T^*M\simeq \R^{2n}$ with $\omega=\omega_M$.
The fact that $\hL_{\e,h}$ is a bisection of $\ha_\e$ and $\hb_\e$ for $h\sim 0$ follows from $\hL_{\e,h}$ being close to $\hL_{\e,0}=\sigma(M)$, which is a bisection for those maps by definition.

Fixing a small enough $h$ together with an auxiliary extra parameter $s\sim \e$, Theorem \ref{thm:1} implies that 
$$\ha_\e\circ \hb_\e|_{\hL_{s,h}}^{-1}:M \to M$$
preserves $\pi$ and Casimirs $C$ up to order $n$ in $\epsilon$. Since the $s$-family $s\mapsto \hL_{s,h}$ consists of Lagrangian bisections, this order of preservation holds for all such $s$, and it thus persists when setting $s=\e$ (the order can only increase). This proves the the first statement.

For the second statement, we need to use the hypothesis that $\hL_{\e,h}$ approximates $L_{\e,h}$.
To analyze the degree of dynamic approximation of the corresponding $\hphi_{\e,h}$, we shall detail the procedure defining $\hphi_{\e,h}$ in a way which will also be useful later for the description of the resulting numerical methods. 

First, by the discussion preceding the statement of this theorem, we can assume
\[L_{\e,h}=\{(x,\theta_{\e,h}(x)):x\in M\}, \ \hL_{\e,h}= \{(x,\htheta_{\e,h}(x)):x\in M\}\]
for smooth $2$-parameter families of closed $1$-forms $\theta_{\e,h},\htheta_{\e,h} \in \Omega^1(M)$. Notice that we have $\theta_{\e,0}=\htheta_{\e,0}=0$ due to $L_{\e,0}=\hL_{\e,0}=\sigma(M)$. Then, we have that 
\begin{equation}\label{eq:hphi}
\hphi_{\e,h}(x_0) = \ha_\e(\hx_{\e,h},\htheta_{\e,h}(\hx_{\e,h}) )
\end{equation}
where $\hx_{\e,h}\in M$ is uniquely determined for $h\sim 0$ by the equation:
\begin{equation}\label{eq:hx}
    \hb_\e(\hx_{\e,h},\htheta_{\e,h}(\hx_{\e,h}))=x_0.
\end{equation}
Note that it follows that $\hx_{\e,0}=x_0$. The exact flow $\varphi_{L_{\e,h}}=\phi^H_{t=\e h}$ induced relative to $R_\e$ is defined in an analogous way with $(\a_\e,\b_\e,\theta_{\e,h})$ in place of $(\ha_\e,\hb_\e,\htheta_{\e,h})$ above.

Finally, the idea is to compare the expansions around $\e=0$ and $h=0$ of eqs. \eqref{eq:hphi}-\eqref{eq:hx} associated with $(\ha_\e,\hb_\e,\htheta_{\e,h})$ with the corresponding expansions for the exact data $(\a_\e,\b_\e,\theta_{\e,h})$. By the hypothesis on the approximate realization data being an order $n$ approximation of $R_\e$, we have
\[ \ha_\e = \sum_{k=0}^n \a_{(k)} \e^k + \mathcal{O}(\e^{n+1})\]
where $\a_{(k)}$ are the Taylor coefficients of $\a_\e$. A similar formula holds for $\hb_\e$ and $\b_\e$. Next, by the hypothesis on $\hL_{\e,h}$ being an order $m$ approximation in $h$ of $L_{\e,h}$, we have
\[ \htheta_{\e,h} = \sum_{l=1}^m \theta_{\e,(l)} h^l+ \mathcal{O}(h^{m+1}), \ \forall \e\sim 0,\]
where $\theta_{\e,(l)}$ are the $m$-first Taylor coefficients of $\theta_{\e,h}$ expanded around $h=0$ for fixed $\e$. (We used $\htheta_{\e,0}=0$ to eliminate the first coefficient.)

It then follows that the Taylor coefficients of the solution $\hx_{\e,h}$ of \eqref{eq:hx} coincide with those of the ``exact'' case defined by $R_\e$ and $L_{\e,h}$ up to terms $\mathcal{O}(\e^{n+1}) + \mathcal{O}(h^{m+1})$. Since $\ha_\e$ approximates $\a_\e$ up to terms $\e^{n+1}$, this estimate persists when applying $\ha_\e$ in \eqref{eq:hphi}. This finishes the proof.
\end{proof}

\begin{remark}[Combined errors]
   The fact that both errors add up is very intuitive and leads to some immediate observations. For instance, one could approximate the geometry to order $n$ large. If the order of approximation of the dynamics, $m$, is lower than $n$, then the dynamic approximation error would dominate, and should not be highly impacted by the error due to the approximation of the geometry.
\end{remark}

\medskip

{\bf Towards concrete numerical methods.} In what follows, we shall consider consider $(M\simeq \R^n, \pi)$ and the Karasev realization data $K_\e$ together the corresponding approximation $\hat K_\e$ of order $n$ given by truncation, as in Section \ref{subsec:Ktruncated}.
Recall that, in this realization, $S\subset T^*M$ is an open and $\omega = \omega_M$. Moreover, since $L_{\e, h}$ is given by a Hamiltonian flow, we can restrict our attention to Lagrangian submanifolds of the form (``type I generating functions'')
\[ graph(dS):=\{ (x, p=dS|_x): x\in M\} \hookrightarrow (T^*M,\omega_M). \]
Note that, in the notation used before Theorem \ref{thm:2}, we have $\theta_L = dS$ for $L=graph(dS)$.
We will thus have 
\begin{equation}\label{eq:Lsgenfun}
L_{\e,h} = graph(dS_{\e,h}), \ \hL_{\e,h} = graph(d\hat S_{\e,h})
\end{equation} 
both described by $2$-parameter families of functions
\[ S_{\e,h}, \hat S_{\e,h}:M \to \R. \]

\begin{remark}[Available tools]
    We observe that, in general, the use of these parametrizations allows us to translate conditions on the Lagrangian submanifolds into PDEs, and use analytic, numerical, or even machine learning techniques to approximate the solutions, see~\cite{vaquero2024designing}.
\end{remark}

\begin{remark}[Rescaling properties from the Karasev realization]\label{rmk:rescalingprop2}
    Recall, from Rmk. \ref{rmk:alpharescaling}, that the Karasev realization data has special rescaling properties. In particular, it follows from the identities in that Remark that
    \[ L_{\e, \lambda h} = \mu_\lambda (L_{\lambda \e, h}) \]
    for any Hamiltonian $H\in C^\infty(M)$.
\end{remark}

%%%%%%%%%%%%%%%%%%%%%%%%%%%%%%%%%%%%%%%%%%%%%%%%%%%%%%
\subsection{Method 1: dynamic approximation via Hamilton-Jacobi}\label{subsection:HJapproximation}
%%%%%%%%%%%%%%%%%%%%%%%%%%%%%%%%%%%%%%%%%%%%%%%%%%%%%%

In this subsection, we produce a concrete numerical method for the ODE \eqref{eq:hameq} by complementing the approximate Karasev realization $\hat K_\e$ of order $n$, for an arbitrary $(M\simeq \R^n,\e\pi)$, with a dynamic approximation $\hL_{\e,h}$ based on the Hamilton-Jacobi method.

\medskip

This dynamic approximation is based on the following general observation. On the one hand, the ``exact'' $L_{\e,h}$ is presented via a generating function as in \eqref{eq:Lsgenfun}. On the other hand, it is given by a Hamiltonian flow of the zero section
\[ L_{\e,h} = \phi^{\a_\e^*H}_h(0_M)\]
inside $(S\subset T^*M,\omega_M=dp_i \wedge dx^i)$. Then, the well-known computation behind the \emph{Hamilton-Jacobi (HJ)} method implies the following HJ-type of evolution PDE\footnote{Note the sign difference on the r.h.s. with respect to some usual HJ-equations due to our convention $\omega_M=-dx^i\wedge dp_i$.} for $S_{\e,h}:M\to \R$:
\begin{equation} \label{eq:HJ}
    (\partial_h S_{\e,h})(x) = (\a_\e^*H)(x,dS_{\e,h}(x)) + c(h),
\end{equation} 
for all $x \in M$ ($M$ connected) and for an arbitrary function $h\mapsto c(h)\in \R$. 
Note that, since $L_{\e,0}=0_M$, we can always fix 
\[S_{\e,0}(x) = 0.\]
We can then proceed again by truncation, as follows. We consider the expansion around $h=0$,
\[S_{\e,h} = \sum_{l=0}^m S_{\e,(l)} h^l + \mathcal{O}(h^{m+1})\]
and notice that the HJ equation \eqref{eq:HJ} fixes uniquely the coefficients $S_{\e,(l)}$ by the recursion obtained after expanding both sides. (See also~\cite[Section $5$]{felemamava17} and \cite{co23}.)
The idea is then to define the approximation $\hL_{\e,h}$ via eq. \eqref{eq:Lsgenfun} with $\hat S_{\e,h}$ given by the truncation of $S_{\e,h}$ up to order $m$ in $h$,
\begin{equation}
    \hat S_{\e,h} := \sum_{l=0}^m S_{\e,(l)} h^l.
\end{equation}
Notice that we can always assume 
\[ \hat S_{\e,(0)}=0, \ \hat S_{\e,(1)} = H. \]

This determines a complete method, which we call {\bf K-HJ Poisson integrator} (after Karasev and Hamilton-Jacobi) with underlying approximation map 
$$M^{KHJ}_{\e,h}(x):= \ha_\e \circ \hb_\e|_{graph(d\hat S_{\e,h})}^{-1}(x).$$
We can follow equations \eqref{eq:hphi} and \eqref{eq:hx} to describe this method as an algorithm.
\begin{algorithm}[H]
\caption{K-HJ Poisson Integrator $x_k \mapsto x_{k+1}$ for the ODE \eqref{eq:hameq}}\label{algorithm:hj}
\begin{algorithmic}[1]

    \State Given $x_k \in \mathbb{R}^n$, compute the unique $\hx$ such that $\hb_\e(\hx,\displaystyle \partial_x\hat S_{\e,h}(\hx)) = x_k$
    \State Compute $x_{k+1}=\ha_\e(\hx,\partial_x\hat S_{\e,h}(\hx))$.
    \end{algorithmic}
\end{algorithm}

\medskip

{\bf Properties of the K-HJ Poisson integrator.} First, notice that, as a direct consequence of Theorem \ref{thm:2}, this method preserves the Poisson structure $\pi$ and any Casimir up to order $n$ in $\e$ and that the overall dynamic approximation order is $\min\{n,m\}$.

\medskip

We further prove that the K-HJ integrator has a special non-trivial property coming from the Karasev realization construction. This remarkable property is quite important in applications, as it shows that only half of the terms of the solution to the Hamilton-Jacobi equation have to be computed. It also provides an explanation for why we observe higher-than-expected orders in some of the numerical simulations of this paper.
\begin{proposition}\label{prop:Sodd}
    In the setting of the K-HJ Poisson integrator, the ``exact'' generating function $S_{\e,h}$ determined by \eqref{eq:HJ}, with any $c(-h)=-c(h)$, is an odd function of $h$:
    \[ S_{\e,-h} = -S_{\e,h}. \]
    In particular, the approximation $\hat S_{\e,h}$ can be taken to have only odd powers of $h$.
\end{proposition}

\begin{proof}
    In this proof, we shall use as an auxiliary the local symplectic groupoid structure $G_\e \rightrightarrows M$ associated with the Karasev realization $K_\e$ of $(M,\e \pi)$, as recalled in Remark \ref{rmk:Klocgd}. In particular, we shall denote $m_\e$ the multiplication map and $inv$ the inversion map. It will be a key point that 
    $$(a): \ inv(x,p)=(x,-p), \ \forall \e,$$ as proven in \cite[Lemma 3.4]{ca22}.

%
%\marginMV{Arriba está $m_\epsilon$ y abajo solo $m$. Habría que homogenizar.}
%
 Let us keep $\e$ fixed and omit it from the notation for a moment. On any local symplectic groupoid, the Hamiltonian flow of $\a^*f$, for any $f\in C^\infty(M)$, is left-invariant for $m$:
 \[ \phi^{\a^*f}_t(z) = m(z, \phi^{\a^*f}_t(\sigma\circ \a(z))), \ z\in G. \]
 It thus follows that the inverse $inv$ with respect to $m$ can be computed through the backward flow:
 \[ inv(\phi^{\a^*f}_t(\sigma(x)) ) = \phi^{\a^*f}_{-t}( \sigma\circ \a(\phi^{\a^*f}_t(\sigma(x) )).\]
Then, denoting $L_0=\sigma(M)$, we conclude that
\[ (b): \ inv( \phi^{\a^*f}_t(L_0) ) = \phi^{\a^*f}_{-t}(L_0), \]
for $t$ small enough so that the flows are defined.

Coming back to our situation, the above relations imply that
\[ L_{\e,-h} = \phi^{\a_\e^*H}_{-h}(L_0) \overset{(b)}{=} inv( \phi^{\a_\e^*H}_{h}(L_0) ) = inv(L_{\e,h} ) \overset{(a)}{=} \mu_{-1}(L_{\e,h}),\]
 where we denoted $\mu_{-1}(x,p)=(x,-p)$ as before. 
 Finally, describing $L_{\e,h}$ with the generating function $S_{\e,h} - c(h)$ as in \eqref{eq:Lsgenfun}, with $c(h)$ odd as in the statement, it follows that $S_{\e,-h}=-S_{\e,h}$, as wanted.
\end{proof}

We finalize this section with another important property of the obtained methods. Given a method to approximate a dynamical system, say $M_h(x)$, the \textbf{adjoint method} is simply defined by
\begin{equation}
M^*_h(x) = M^{-1}_{-h}(x). %\tag{Adjoint of a Method}    
\end{equation}
Notice that since $M_h(x)$ is not a flow in general, then   $M^{-1}_{-h}(x) \neq M_h(x)$ and the equation above usually produces different methods. The following result is a direct consequence for the K-HJ method of the above special property.
\begin{corollary}[K-HJ integrators are Self-Adjoint] The K-HJ Poisson integrator is self-adjoint in $h$:
\[ \left(M^{KHJ}_{\e,h}\right)^* := (M^{KHJ}_{\e,-h})^{-1} = M^{KHJ}_{\e,h}. \]
\end{corollary}

\medskip

%%%%%%%%%%%%%%%%%%%%%%%%%%%%%%%%%%%%%%%%%%%%%%%%%%%%%%
\subsection{Method 2: dynamic approximation via Collective Integrators}\label{subsec:collective}
%%%%%%%%%%%%%%%%%%%%%%%%%%%%%%%%%%%%%%%%%%%%%%%%%%%%%%
Another natural way to approach our dynamic approximation step is through the direct use of {\it collective integrators} to approximate $\phi^{\ha_\e^*H}_h$, see~\cite{McLachlan_2014}. 

To explain this approximation, let us recall that the exact flow for our ODE \eqref{eq:hameq} can be obtained as in Thm. \ref{CDW} together with Rmk. \ref{rmk:rescaling},
\[ \phi^H_{t=\e h} = \a_\e \circ \phi^{\a_\e^*H}_h \circ \sigma .\]
We already know that we can replace $\a_\e$ with the Karasev approximate realization $\ha_\e$, leading to a \emph{preliminary approximation map}:
\[ \hphi_{\e,h} = \ha_\e \circ  \phi^{\ha_\e^*H}_h \circ \sigma: M \to M. \]
Note that this is approximately in the general class discussed at the beginning of this Section, since
\[ \hphi_{\e,h} =  \halpha_\epsilon \circ \hbeta_\epsilon|_{\hL_{\e,h}}^{-1} \ mod \ \mathcal{O}(\e^{n+1}), \text{ for } \hL_{\e,h}=\phi^{\ha_\e^*H}_h(0_M).\]
This follows from the fact that $\hb_\e$ is approximately conserved along the Hamiltonian flow of $\ha_\e^*H$, since $\b_\e$ is exactly conserved along the flow of $\a_\e^*H$ (see the observation below Theorem \ref{CDW}).

\medskip

Nevertheless, the above is not enough to define a complete method since the ``exact'' Hamiltonian flow 
 of $\ha_\e^*H$ can be as hard to compute as for the original ODE \eqref{eq:hameq}. We thus consider an approximation $$\Sigma^{\ha_\e^*H}_h:\R^{2n}\to \R^{2n}$$ for the Hamiltonian flow $\phi^{\ha_\e^*H}_h$ defined by a \emph{symplectic integrator of order $m$ in $h$}, $\Sigma_h$, on $(S\subset \R^{2n},\omega=dp_i\wedge dx^i)$ applied to the \emph{collective Hamiltonian} $\ha_\e^*H \in C^\infty(\R^{2n})$.
The corresponding complete method will be called {\bf K-Collective Poisson integrator}, and it has as underlying approximation map
\[ M^{KC}_{\e,h}(x) : = \ha_\e \circ \Sigma^{\ha_\e^*H}_h \circ \sigma .\]
As an algorithm, we have the following description.
\begin{algorithm}[H]
\caption{K-Collective Poisson Integrator $x_k \mapsto x_{k+1}$ for the ODE \eqref{eq:hameq} }\label{algorithm:collective}
\begin{algorithmic}[1]
%\Procedure{Roy}{$a,b$}      % \Comment{This is a test}
    \State Given $x_k \in \mathbb{R}^n$,
    evolve $(x_k,p_k=0)\in \R^{2n}$ using the symplectic integrator $\Sigma_h$, with time step $h$ and Hamiltonian $\ha_\e^*H$, to obtain $(\hat x_{k+1},\hat p_{k+1})$
    \State Compute $x_{k+1} = \ha_\e (\hat x_{k+1},\hat p_{k+1})$.
    \end{algorithmic}
\end{algorithm}
Notice that, in general, this method is not directly in the class of methods defined by $\hat K_\e$, together with a choice of Lagrangian $\hL_{\e,h}$, as in the setting of Theorem \ref{thm:2}. Nevertheless, it is approximately so and we can deduce its properties, as follows.

\begin{proposition}
    The K-Collective Poisson integrator, defined by $\hat K_\e$ of order $n$ and a symplectic integrator $\Sigma_h$ of order $m$ in $h$, coincides with the map
    \[
    \halpha_\epsilon \circ \hbeta_\epsilon|_{\hL_{\e,h}}^{-1}  \text{ for } \hL_{\e,h}=\phi^{\ha_\e^*H}_h(0_M)\]
    up to order $n$ in $\e$ and up to order $m$ in $h$. In particular, as a consequence of Theorem \ref{thm:2}, it preserves the Poisson geometry to order $n$ and it approximates the dynamics to order $\min\{n,m\}$.
\end{proposition}
\begin{proof}
Before the statement of the theorem, we already explained that the map $\halpha_\epsilon \circ \hbeta_\epsilon|_{\hL_{\e,h}}^{-1}$ coincides with $\hphi_{\e,h} = \ha_\e \circ  \phi^{\ha_\e^*H}_h \circ \sigma$ up to order $n$ in $\e$. It is then enough to show that $M^{KC}_{\e,h}$ coincides with the latter $\hphi_{\e,h}$ up to order $m$ in $h$.
But this follows directly from the fact that the symplectic integrator is of order $m$, so that $\Sigma_h^{\ha_\e^*H}=\phi^{\ha_\e^*H}_h$ modulo $\mathcal{O}(h^{m+1})$ errors.
\end{proof}

\begin{remark}[Rescaling properties from those of $K_\e$]
Let $\hL_{\e,h} = \phi^{\ha_\e^*H}_h(0_M)$, where the Hamiltonian flow is the ``exact'' one. By the rescaling properties of the approximate $\ha_\e$ observed in Remark \ref{rmk:rescalingKaprox}, it follows that
\[ \hL_{\e,\lambda h} = \mu_\lambda(\hL_{\lambda \e, h}) \text{ and } \ha_{\e}\circ \phi^{\ha_{\e}^*H}_{\lambda h} \circ \mu_\lambda = \ha_{\lambda \e}\circ \phi^{\ha_{\lambda \e}^*H}_h  . \]
Underlying these properties, on top of the rescaling one $\ha_{\lambda \e}=\ha_\e \circ \mu_\lambda$, one uses the rescaling property $\mu_\lambda^* \omega = \lambda \omega$ of the symplectic form and its consequences for Hamiltonian flows.
On the other hand, since the symplectic integrator is of order $m$, an analogue of the eq. in Remark \ref{rmk:rescalingKaprox} holds approximately for $\Sigma^{\ha_\e^*H}_h$, leading to the following \emph{approximate rescaling property of the K-Collective method}:
\[ M^{KC}_{ \e, \lambda h} = M^{KC}_{\lambda \e,h } \text{ up to order $m$ in $h$.} \] 
\end{remark}

%%%%%%%%%%%%%%%%%%%%%%%%%%%%%%%%%%%%%%%%%%%%%%%%%%%%%%
%%%%%%%%%%%%%%%%%%%%%%%%%%%%%%%%%%%%%%%%%%%%%%%%%%%%%%
\section{Simulations and Numerical Experimentation}\label{sec:simulations}
%%%%%%%%%%%%%%%%%%%%%%%%%%%%%%%%%%%%%%%%%%%%%%%%%%%%%%
%%%%%%%%%%%%%%%%%%%%%%%%%%%%%%%%%%%%%%%%%%%%%%%%%%%%%%

In this section we illustrate our finding in several examples. Namely,

\begin{itemize}
    \item $\mathfrak{so}^*(3)$ {\it and the Rigid Body}. In this case we have a linear Poisson structure and the integrating symplectic groupoid is known (see \cite{felemamava17} and references therein). Nonetheless, we approximate the symplectic groupoid using our constructions and test the obtained results.
    \item {\it Lotka-Volterra Dynamics}. In this case we have a quadratic Poisson tensor over $\mathbb{R}^3$ with a Linear Hamiltonian. The dynamics are easy to approximate, but the approximation of the symplectic groupoid is harder than in $\mathfrak{so}^*(3)$.
    \item {\it Non-canonical Symplectic Structure}. Our constructions apply not only to Poisson manifolds, but they can also used for symplectic structures that are non-canonical. Here we study a simple example where we consider a canonical symplectic structure modified by a magnetic term.
\end{itemize}

We focus on the first method, which we called K-HJ Poisson integrator, and relegate the second (K-Collective) for future work.
Then, in the examples below, we shall be considering the K-HJ method $M^{KHJ}_{\e,h}$, where the underlying approximation $\ha_\e$ of Section \ref{subsec:Ktruncated} is constructed by numerically approximating the map $\phi_\e$ in the defining relation \eqref{eq:alphaKar} for $\a_\e$ to a high degree and, then, after expanding both sides in $\e$, solving for the coefficients $\a_{(k)}, \ k\leq n$, symbolically using \emph{Mathematica}.

%%%%%%%%%%%%%%%%%%%%%%%%%%%%%%%%%%%%%%%%%%%%%%%%%%%%%%
\subsection{Example 1: $\mathfrak{so}^*(3)$ and the rigid body}
%%%%%%%%%%%%%%%%%%%%%%%%%%%%%%%%%%%%%%%%%%%%%%%%%%%%%%
It is well known that the symplectic groupoid integrating $\mathfrak{so}^*(3)$ is just $T^*SO(3)$. Nonetheless, we illustrate here our findings providing an alternative construction of the integrating symplectic groupoid. A first question is how the error is reflected on the approximation of the symplectic realization $\alpha_\epsilon$. Here, we illustrate the answer by using as a measurement of the error the discrepancy between the Poisson tensor acting on the basis $dx^1\wedge dx^2$, $dx^2\wedge dx^3$ and $dx^1\wedge dx^3$ and the symplectic structure acting on its pull-backed version. Namely, if we denote by $\omega$ the canonical symplectic form on $\mathbb{R}^6$,
% \begin{align*}
%     Error^n(\epsilon) = &\left((\epsilon\pi(dx^1, dx^2) - \Omega((\hat\alpha^{n}_\epsilon)^*dx^1,(\hat\alpha^{n}_\epsilon)^*dx^2))^2 \right.
%     \\ &+ (\epsilon\pi(dx^1, dx^3) - \Omega((\hat\alpha^{n}_\epsilon)^*dx^1,(\hat\alpha^{n}_\epsilon)^*dx^3))^2 \\
%      & \left.+ (\epsilon\pi(dx^3, dx^2) - \Omega((\hat\alpha^{n}_\epsilon)^*dx^3,(\hat\alpha^{n}_\epsilon)^*dx^2))^2 \right)^{1/2}
% \end{align*} 

\begin{align}\label{sym:error}
    Error^n(\epsilon) = &[(\epsilon\pi(dx^1, dx^2) - \omega(\hat\alpha_\epsilon^*dx^1,\hat\alpha_\epsilon^*dx^2))^2 \notag
    \\ &+ (\epsilon\pi(dx^1, dx^3) - \omega(\hat\alpha_\epsilon^*dx^1,\hat\alpha_\epsilon^*dx^3))^2 \notag \\
     & + (\epsilon\pi(dx^3, dx^2) - \omega(\hat\alpha_\epsilon^*dx^3,\hat\alpha_\epsilon^*dx^2))^2 ]^{1/2}. \tag{Symplectic Realization Error}
\end{align} 

Figure~\ref{fig:so(2)} below plots the symplectic realization error as a function of $\epsilon$.

\vspace*{0.5cm}
\begin{figure}[H]
    \centering
    \includegraphics[width = 0.7\textwidth]{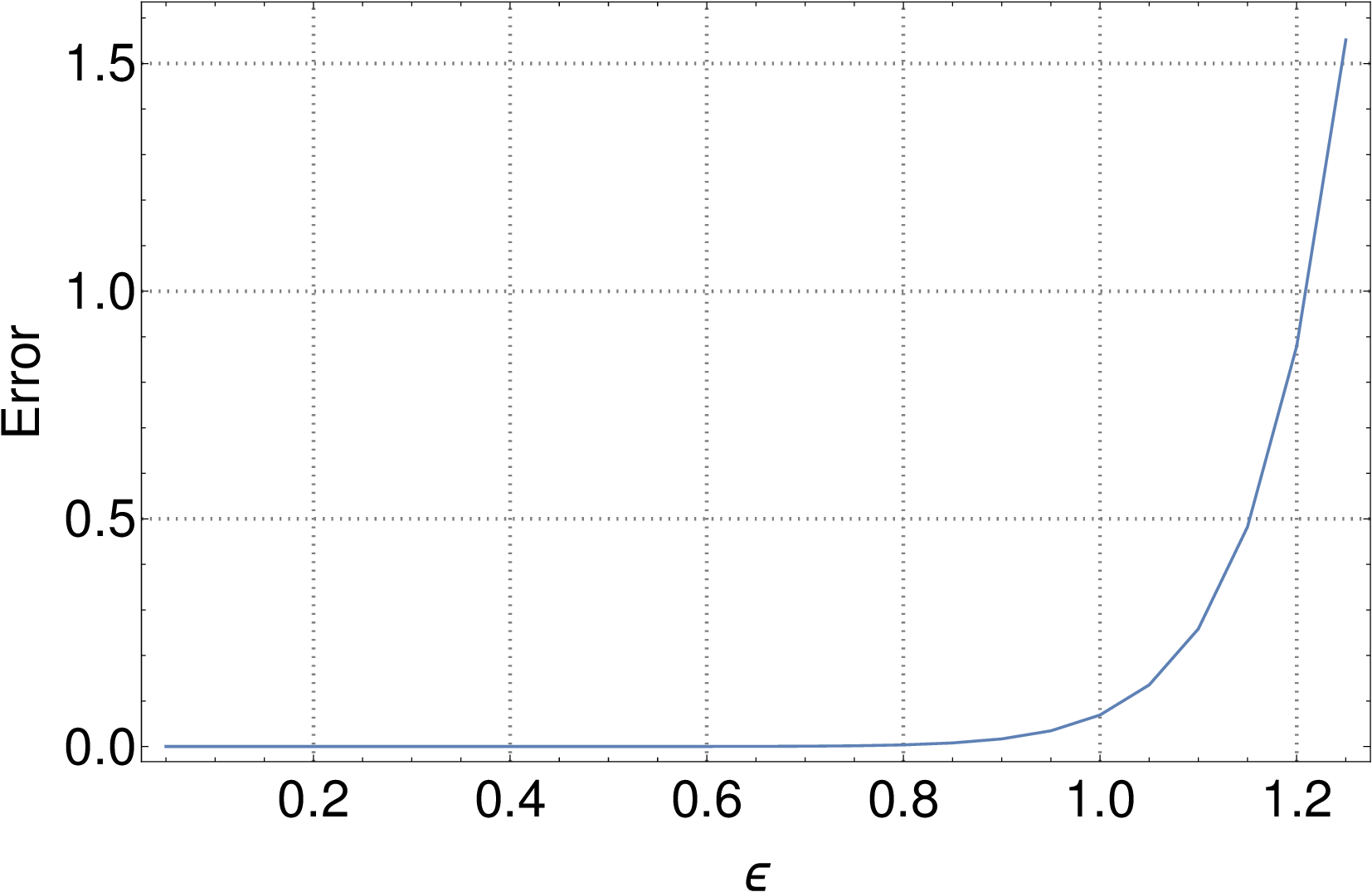}
    \includegraphics[width = 0.7\textwidth]{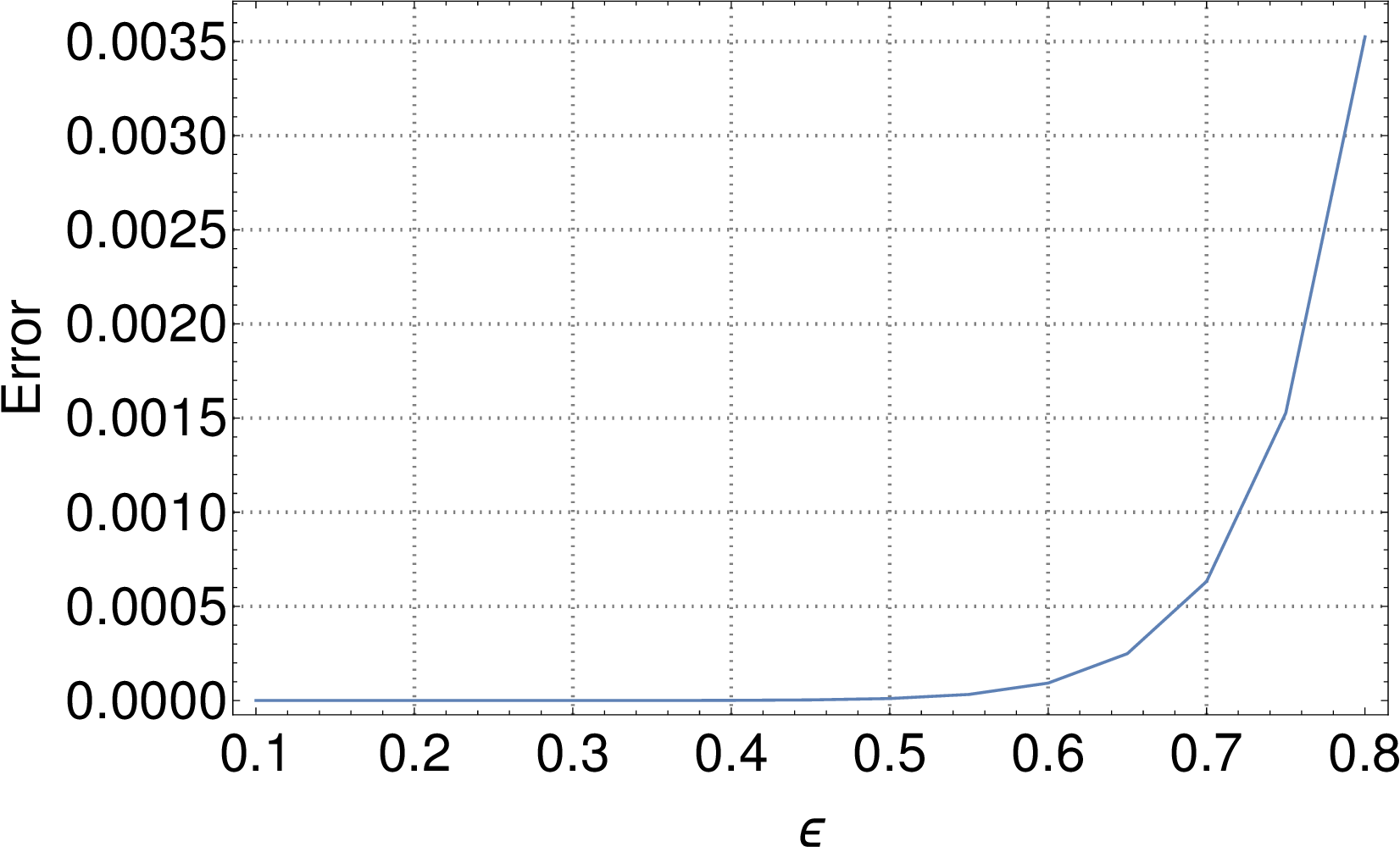}
\end{figure}

\begin{figure}[H]
    \centering
    \includegraphics[width = 0.6\textwidth]{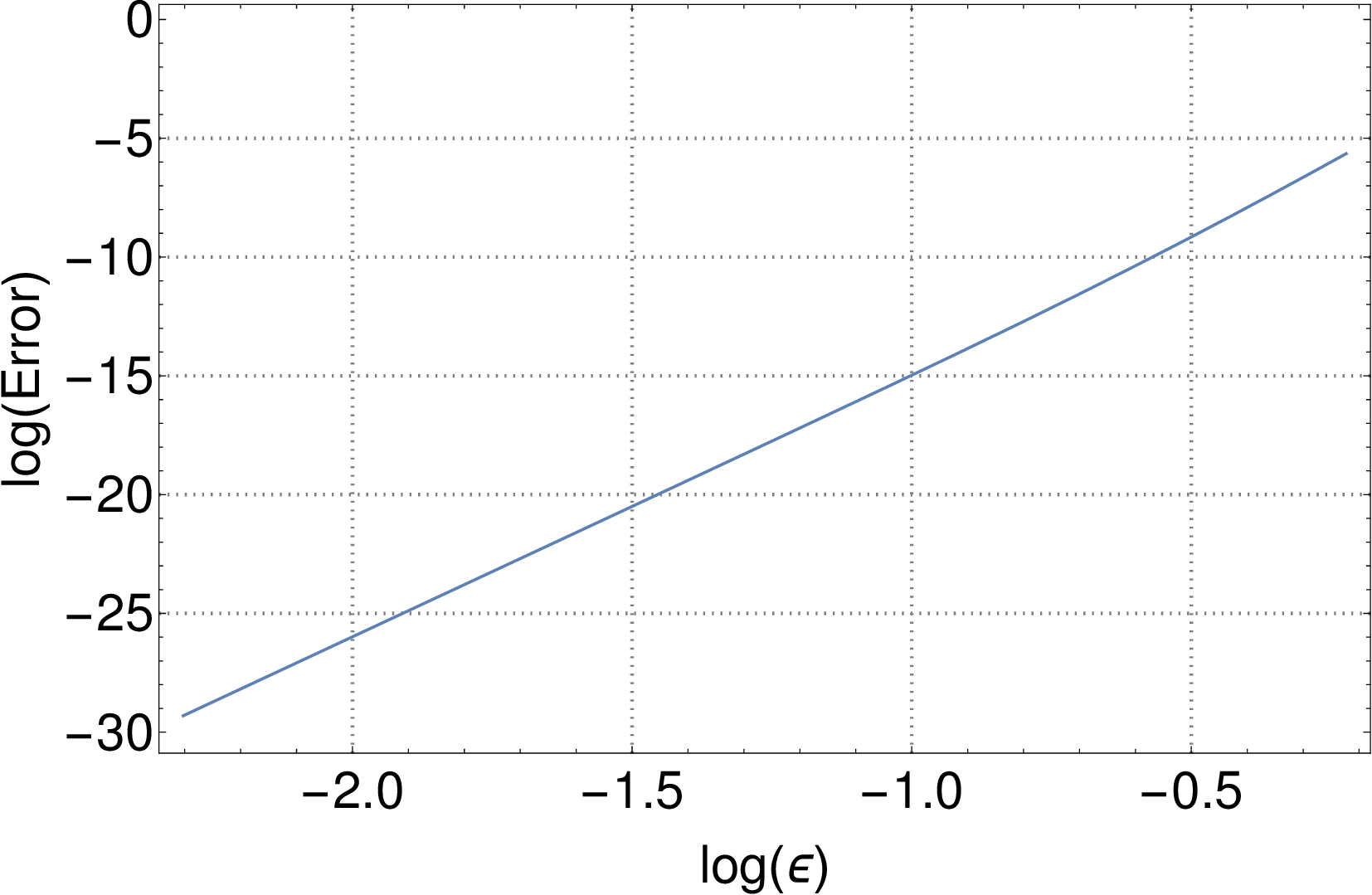}
    \caption{Illustration of the error of the symplectic realization at the point $(2, 3, 3, 1, 2, 3)$. ({\it Top}) As explained in the beginning of this section, we consider an order $n=10$ approximation $\ha_\e$ to $\alpha_\epsilon$ computed using Mathematica. We showcase the error, measured as the square of the differences in the realization, following~\eqref{sym:error}. We observe a small error even for large values of $\epsilon$. ({\it Middle}) The same plot but restricted to the interval $[0,0.8]$, where the error shows the biggest increment. We see that up to $\epsilon = 0.6$ the error is quite small. ({\it Bottom}) Representation of the logarithm of the error versus the logarithm of epsilon. We observe numerically a slope of approximately $10.99$ for small values of $\epsilon$. Taking into account that derivatives have been computed numerically, which introduces an error, this matches our theoretical prediction of the error being $\mathcal{O}(\epsilon^{n +1}) = \mathcal{O}(\epsilon^{11})$.
    }
    \label{fig:so(2)}
\end{figure}
%
% Taken from the Mathematica code: General_poisson_integrators_order_10.
%
Since $n = 10$ is a quite large order of approximation, we consider below a more modest scenario, where we take $n = 4$. The corresponding plots are presented below.

\begin{figure}[H]
    \centering
    \includegraphics[width = 0.6\textwidth]{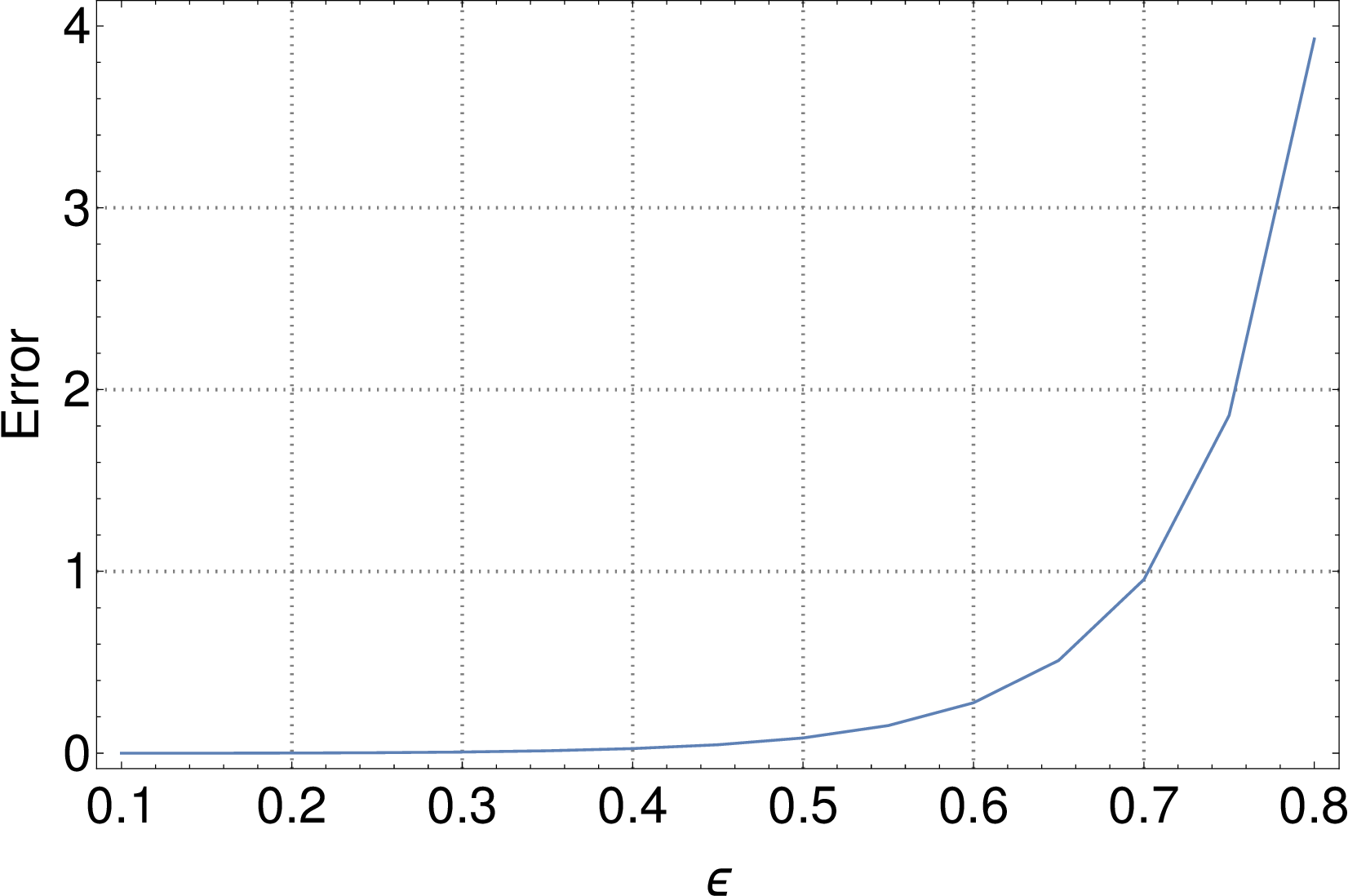}
    \includegraphics[width = 0.6\textwidth]{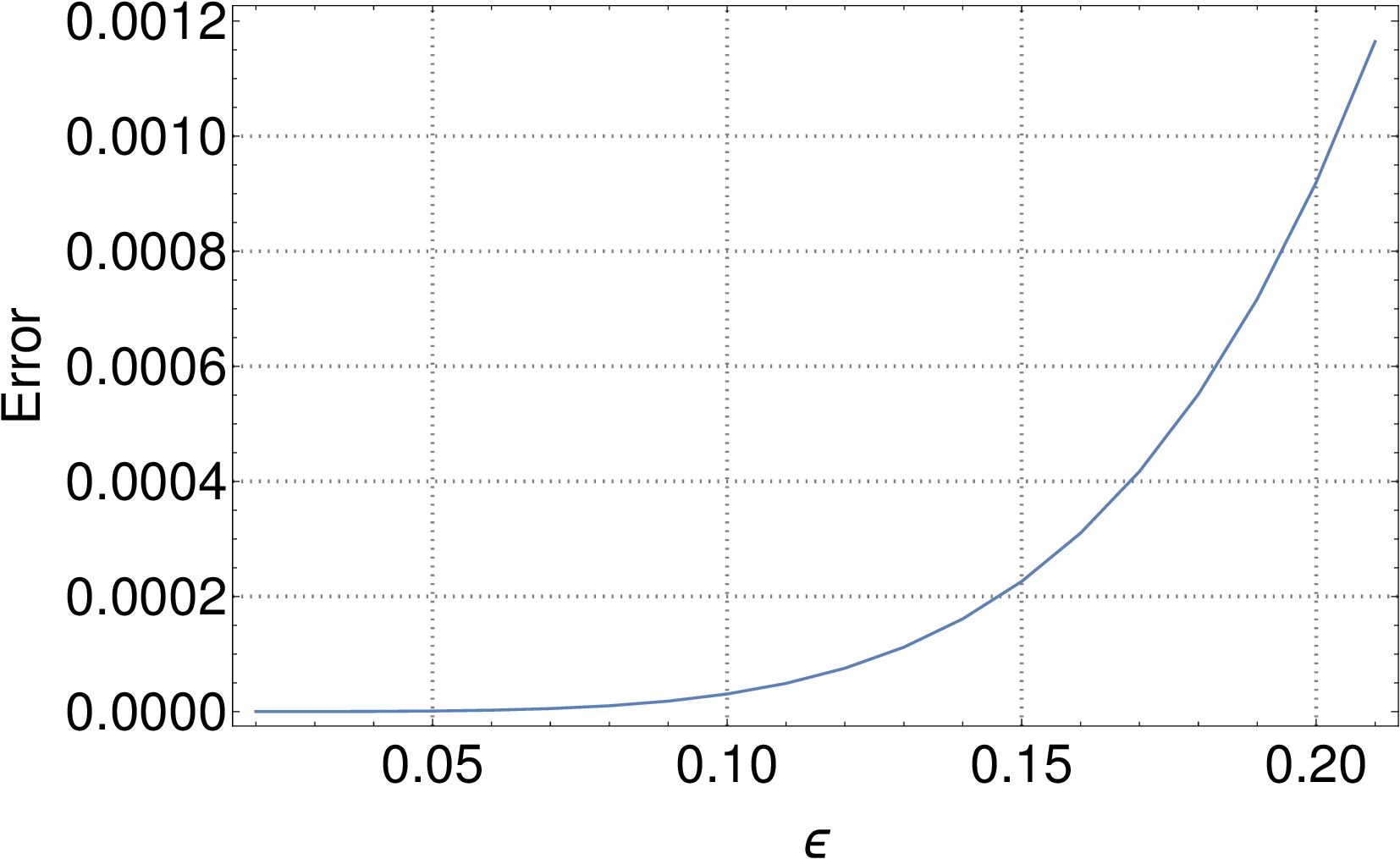}
      \includegraphics[width = 0.6\textwidth]{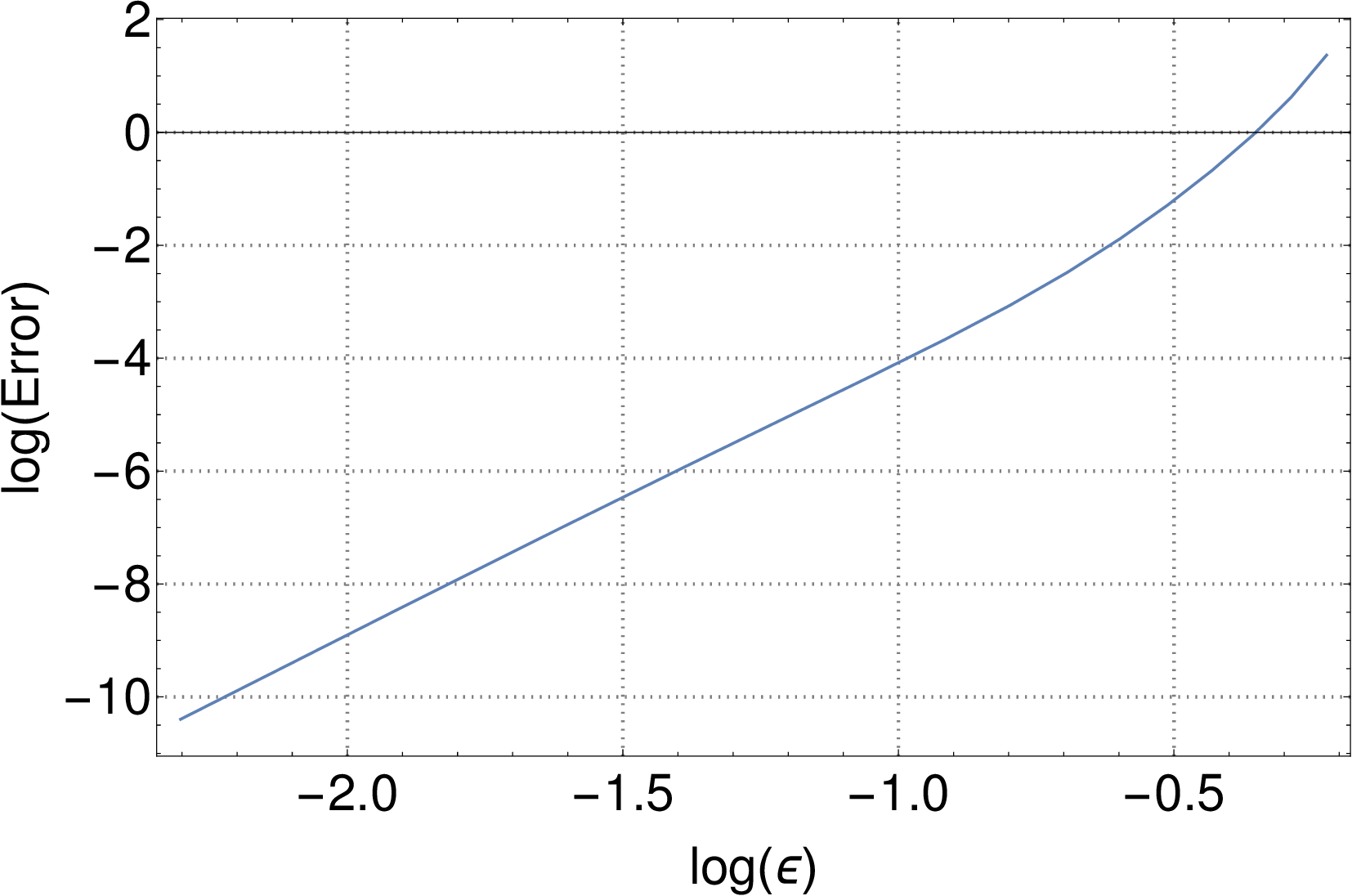}
    \caption{({\it Top}) We consider an order $4$ approximation to $\alpha_\epsilon$. We observe a small error even for quite large values of $\epsilon$. ({\it Middle}) The same plot as in the figure above, but with $\epsilon$ ranging on the interval $[0,0.2]$. ({\it Bottom}) Plot of the figure above, but representing the the logarithm of the error versus the logarithm of epsilon. We observe numerically a slope of approximately $4.9$ for small values of $\epsilon$. Taking into account that derivatives have been computed numerically, this matches our theoretical prediction of the error being $\mathcal{O}(\epsilon^{n +1}) = \mathcal{O}(\epsilon^{5})$.
    }
    \label{fig:so(23}
\end{figure}
%
% Taken from the Mathematica code: General_poisson_integrators_order_4.
%
The figures above match our theoretical findings and showcase nice approximations of the symplectic realization. We proceed now to study to what extend these nice properties also hold for the approximation of the dynamics. We study the evolution of  the Casimir and Hamiltonian, when the Hamiltonian is given by
\[H = (x^1)^2/2 + (x^2)^2/1.5 + (x^3)^2/2.5.\]
We take values $\epsilon = 0.1$, stepsize $h = 2$, and solve the Hamilton-Jacobi  to order $2$. In our setting $\alpha_\epsilon$ is approximated to order $6$. We obtain then integrators and observe the conservation of the Hamiltonian and Casimir ($Casimir(x^1,x^2,x^3) = (x^1)^2 + (x^2)^2 + (x^3)^2$) through the evolution across a simulated trajectory starting at the point $(1,3,3)$ for reference.
\begin{figure}[H]
    \centering
    \includegraphics[width = 0.6\textwidth]{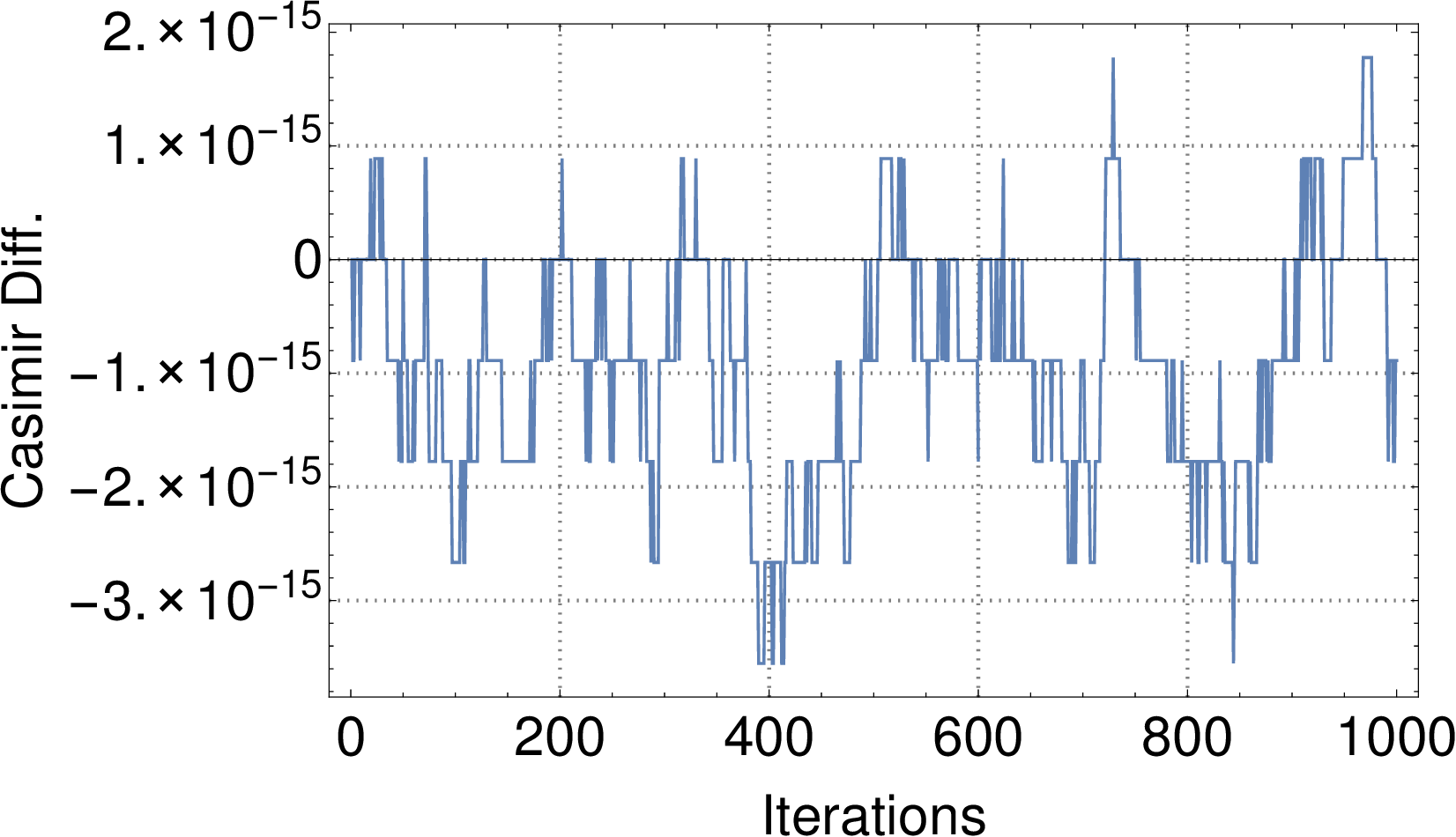}
    \includegraphics[width = 0.6\textwidth]{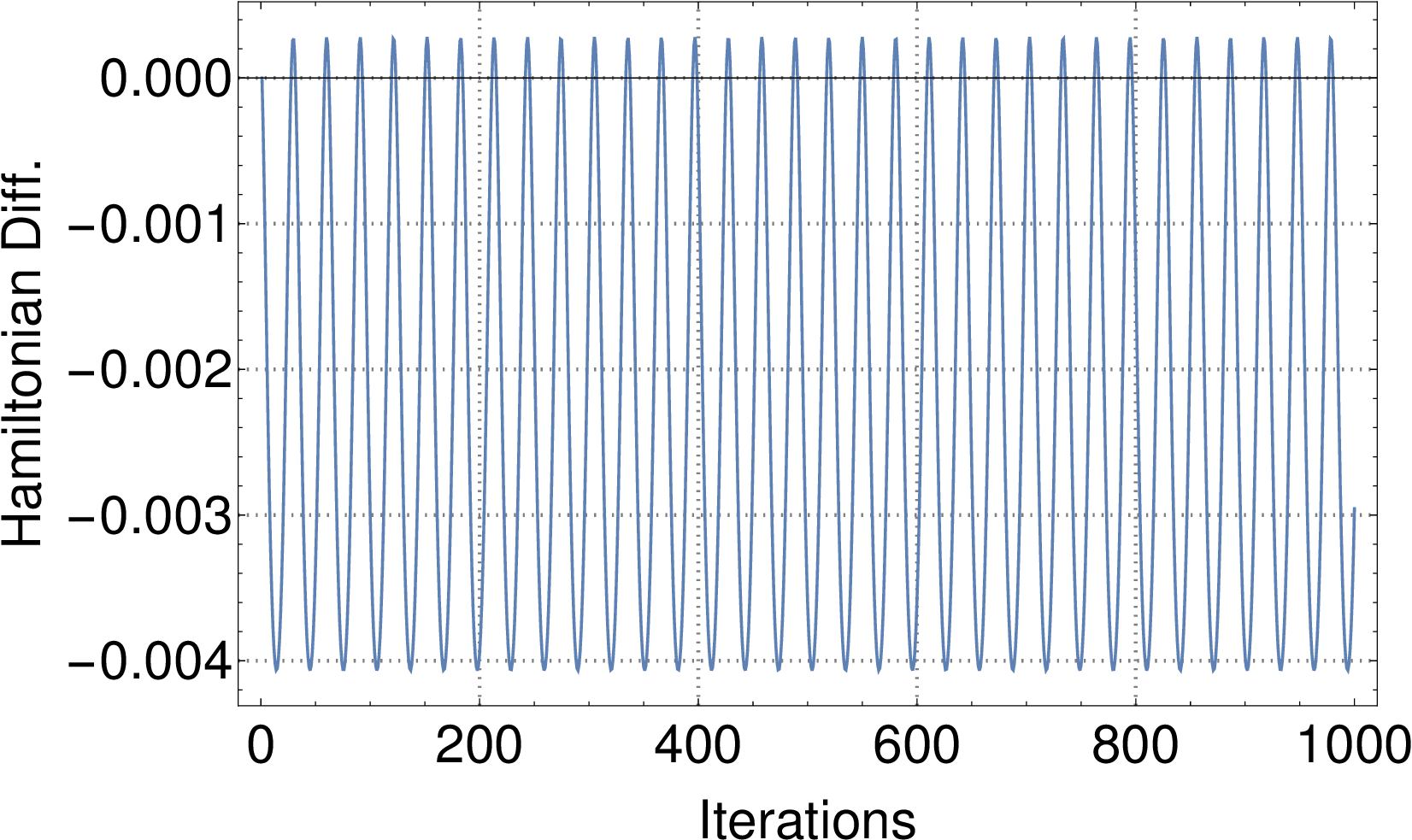}
    \caption{({\it Top}) Evolution of the difference of the Casimir (value at iteration - initial value). We observe conservation of the Casimir of order higher than expected theoretically. This is probably due to the fact that the Poisson structure is linear.  ({\it Bottom}) Evolution of the difference of the Hamiltonian (value at current iteration - initial value) through the simulated trajectory. Due to the low order approximation of the dynamics we observe oscillations of the expected order. These plots clearly illustrate how dynamics and geometry can be approximated to different orders.  }
    \label{fig:rigid_body_dynamics}
\end{figure}

%
% Taken from the Mathematica code: formal_so3_rigid_body.nb.
%

In order to compare our results with other non-geometric integrators, we decided to run analogous simulations using ODE45 integrator. We follow an implementation of ODE45 for Mathematica taken from~\cite{stack}, and we show the results in the plots below.
\begin{figure}[H]
    \centering
        \includegraphics[width = 0.65\textwidth]{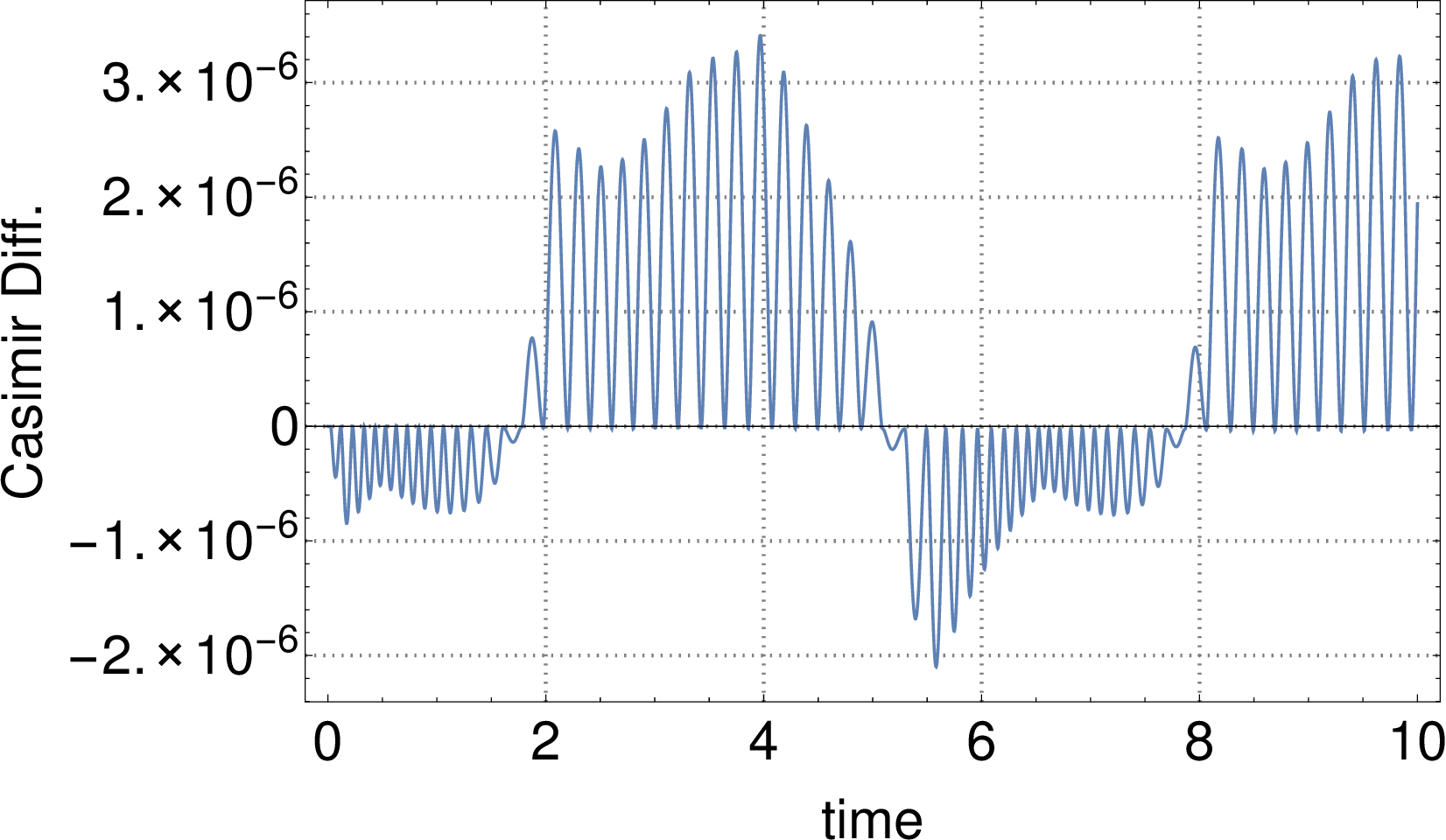}
  \includegraphics[width = 0.65\textwidth]{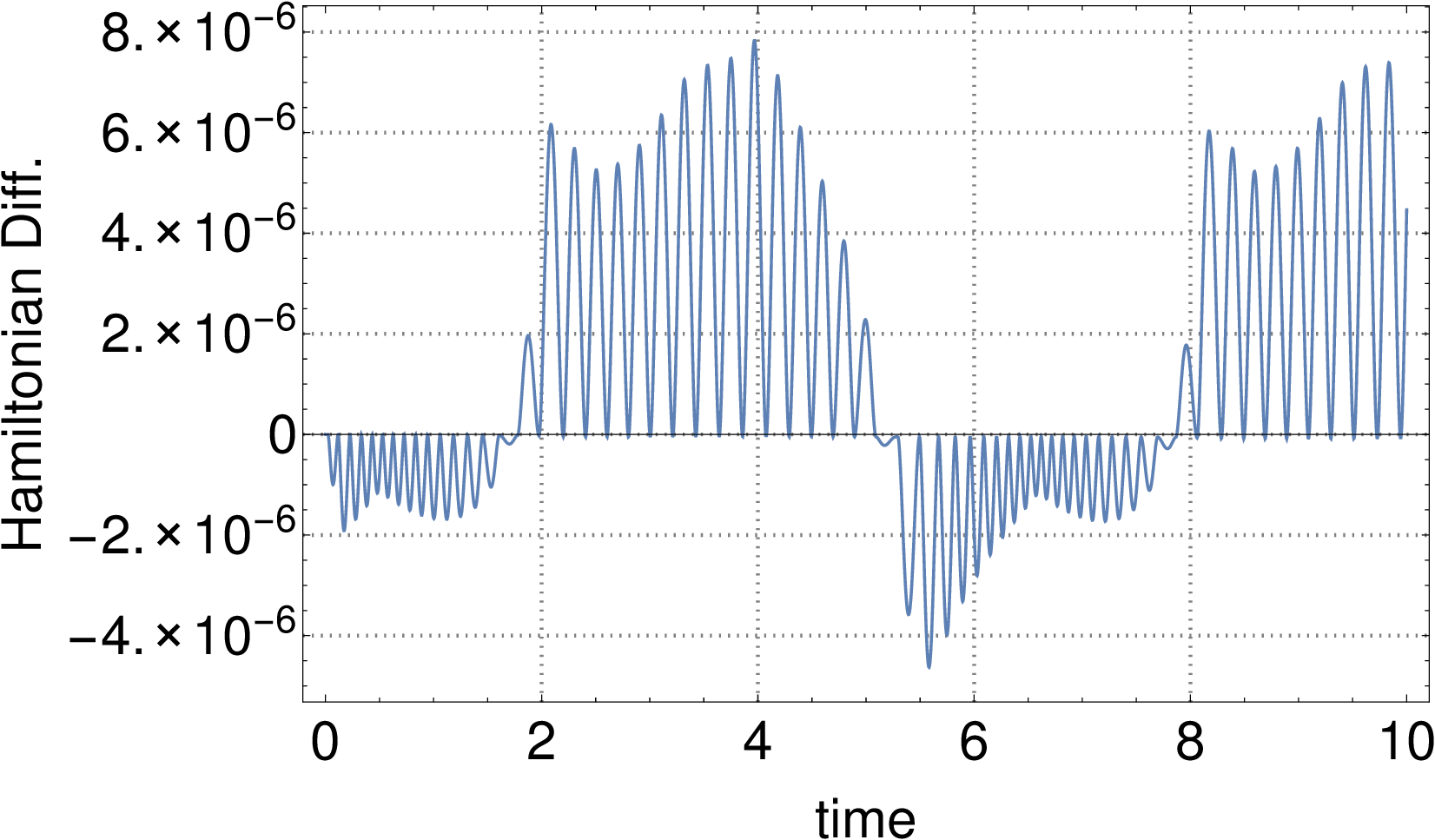}
    \caption{Evolution of Casimir  and Hamiltonian difference throughout a trajectory of ODE45. In this case, dynamics and geometry are conserved at a similar order (order $5$).}
\label{fig:rigid_body_dynamics}
\end{figure}

%
% Taken from the Mathematica code: formal_so3_rigid_body.nb.
%

Finally, we study the conservation of the Poisson tensor through Lagrangian bisections when approximations of the source and target mappings are used. We use the transformation provided by the integrator obtained using Hamitlon-Jacobi in the previous sections and plot the error as a function of $\epsilon$. In this case, once a Lagrangian bisection $L$ is fixed, the error is defined as
\begin{equation}\label{error} \tag{Tensor Cons. Error}
\begin{aligned}
    Error^n(\epsilon) = &\left((\pi(dx^1, dx^2) - \pi((\hphi_{\e,h})^*dx^1, (\hphi_{\e,h})^*dx^2))^2 \right. \notag
    \\ &+ (\pi(dx^1, dx^3) - \pi((\hphi_{\e,h})^*dx^1, (\hphi_{\e,h})^*dx^3))^2  \notag \\
     & \left.+ (\pi(dx^3, dx^2) - \pi((\hphi_{\e,h})^*dx^3, (\hphi_{\e,h})^*dx^2))^2 \right)^{1/2} .
\end{aligned}
\end{equation}
Finally, we approximate the source and target to order $6$ and fix $L$, the solution of the Hamilton-Jacobi equation obtained in the previous section. The plot of error against $\epsilon$ is illustrated in Figure~\ref{fig:rigid_body_poisson_error}.
\begin{figure}[H]
    \centering
    \includegraphics[width = 0.7\textwidth]{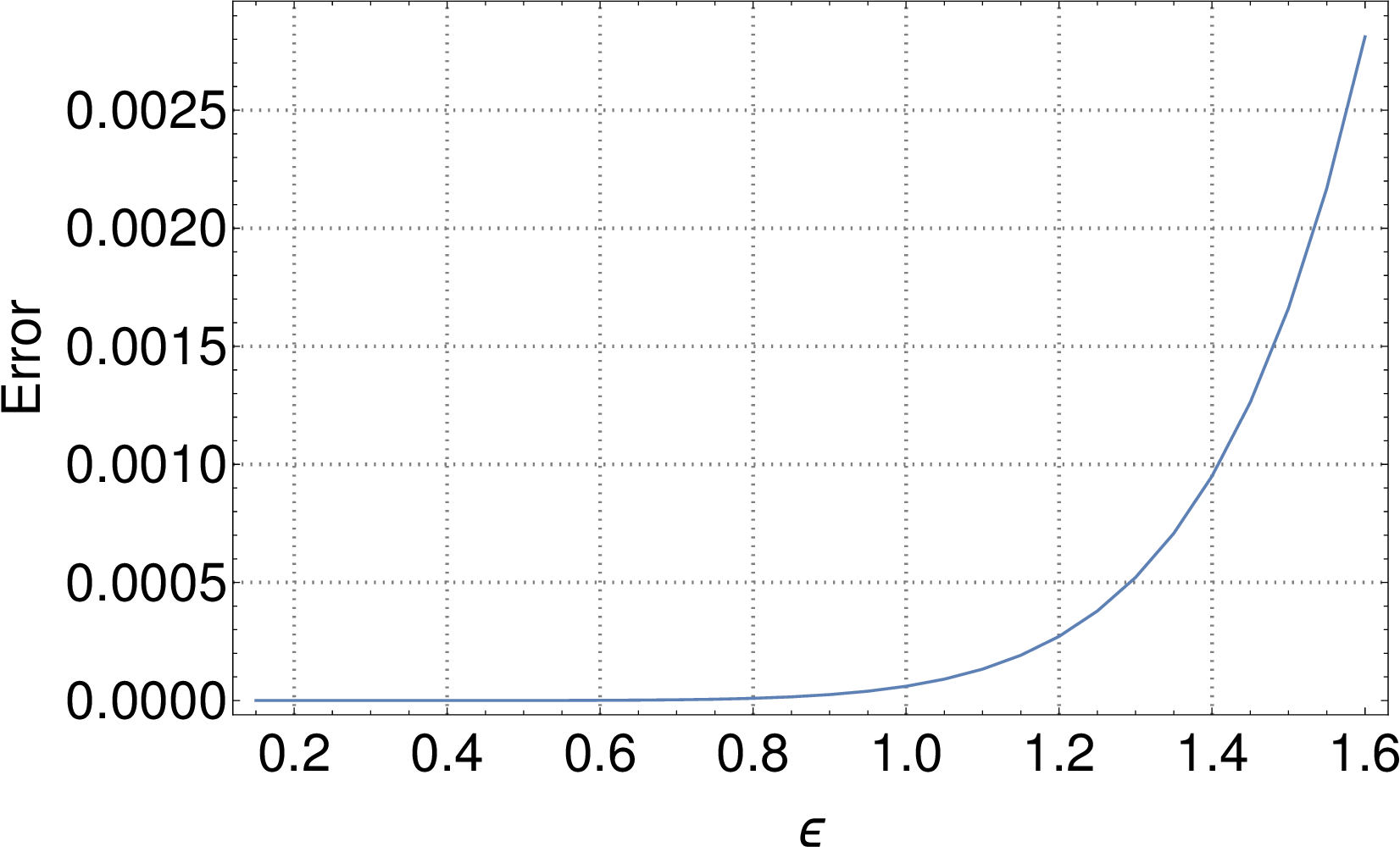}
    \includegraphics[width = 0.7\textwidth]{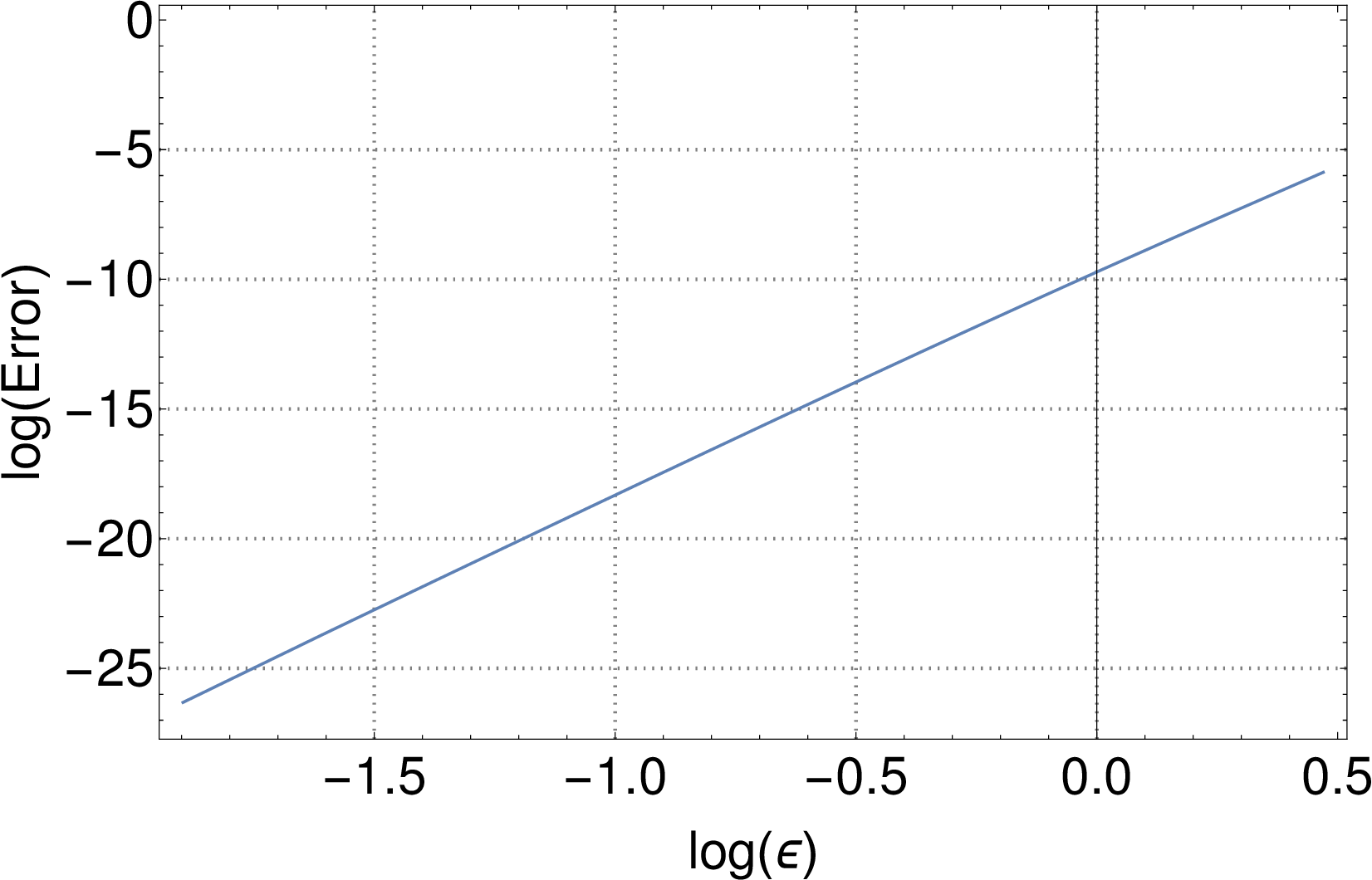}
    \caption{({\it Top}) Plot of the tensor conservation error as a function of epsilon. We observe a very small error even for large values of $\epsilon$. ({\it Bottom}) Plot of the same data as in the figure above, but representing the logarithm of $\epsilon$ versus the logarithm of the error. We observe a slope even larger than the expected one. This is probably due to the fact that the Poisson structure is linear and henceforth easy to approximate.}
    \label{fig:rigid_body_poisson_error}
\end{figure}
%
% Taken from the Mathematica code: formal_so3_rigid_body_error_poisson_tensor.nb.
%

%%%%%%%%%%%%%%%%%%%%%%%%%%%%%%%%%%%%%%%%%%%%%%%%%%%%%%
\subsection{Example $2$: Lotka-Volterra dynamics}
%%%%%%%%%%%%%%%%%%%%%%%%%%%%%%%%%%%%%%%%%%%%%%%%%%%%%%

This example is taken from \cite{colasal23}, where the reader can find more details. The setting is given by $M=\R^3\ni (x_1,x_2,x_3)$ with quadratic Poisson structure 
\[\pi = 2 x_1x_2 \partial_{x_1}\wedge \partial_{x_2} + 2 x_1x_3 \partial_{x_1}\wedge \partial_{x_3} + 2 x_2x_3 \partial_{x_2}\wedge \partial_{x_3}\]
and Hamiltonian
\[
H(x^1,x^2,x^3) = x^1 + x^2 + x^3.
\]
We take
$\epsilon = 0.1$, $h = 2$, and solve the Hamilton-Jacobi  to second $2$. The mapping $\alpha_\epsilon$  is approximated to order $8$. We simulate a trajectory starting from the point $(1,3,3)$.
\begin{figure}[H]
    \centering
    \includegraphics[width = 0.7\textwidth]{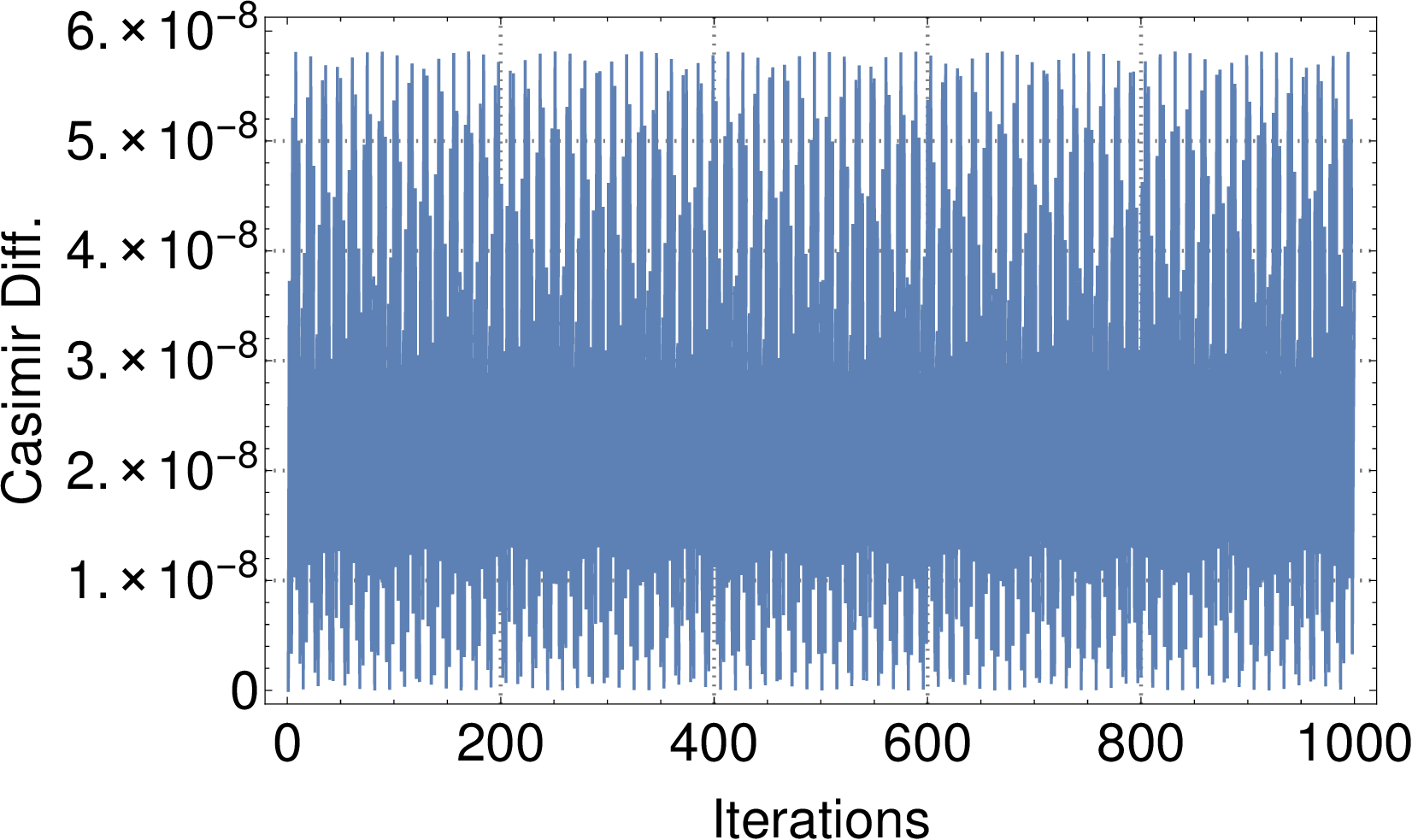}
    \includegraphics[width = 0.7\textwidth]{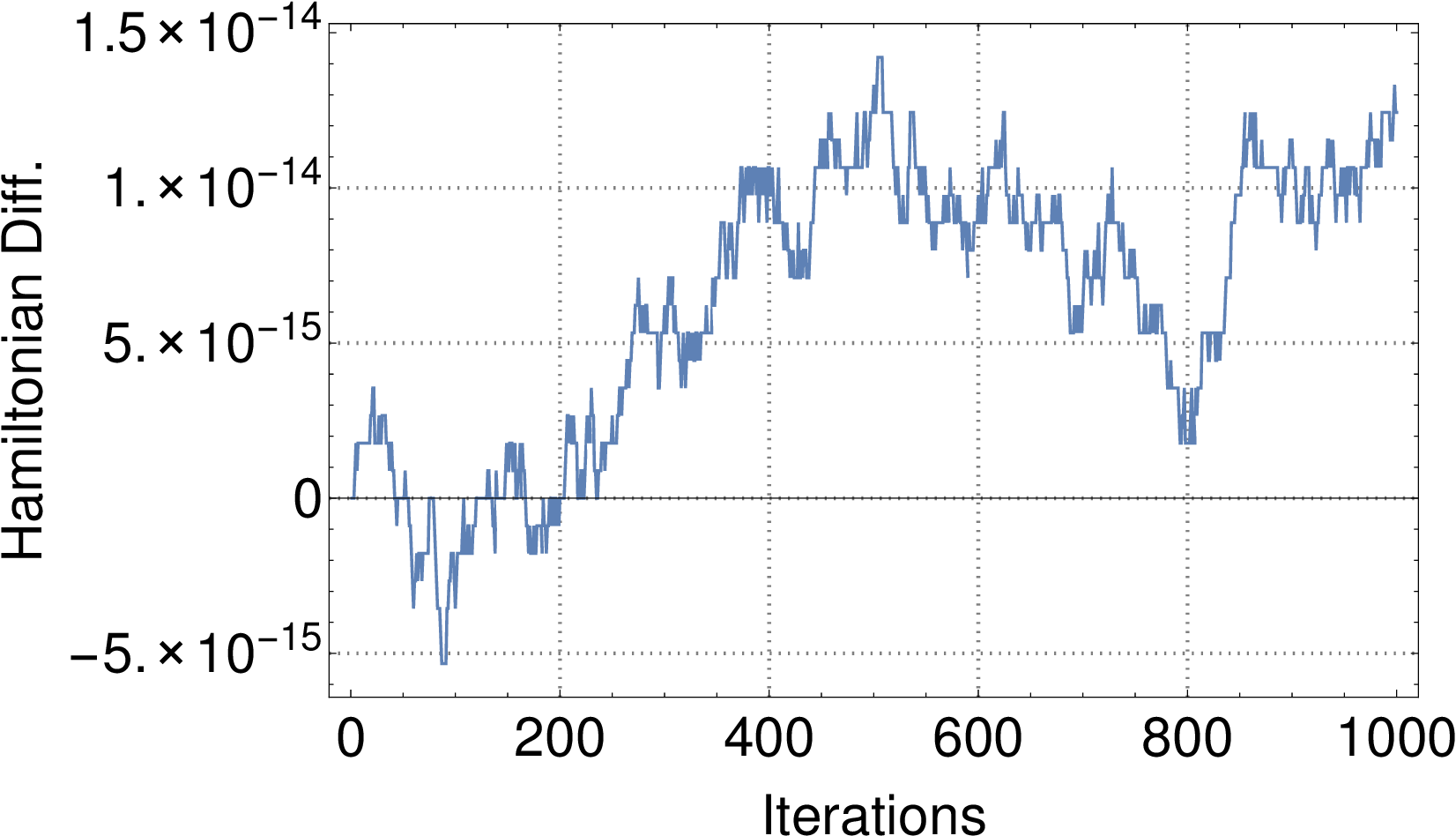}
    \caption{{\it (Top)} Evolution of the difference of Casimir values (value at iteration - initial value) through the simulated trajectory starting at $(1,3,3)$.  ({\it Bottom}) Evolution of the difference of the Hamiltonian across the simulated trajectory.}
    \label{fig:rigid_body_dynamics}
\end{figure}

%
% Taken from the Mathematica code: formal_lotkaVolterra_example_1_casimir_diff_hamiltonian_diff.
%

In order to illustrate the impact of the chosen value of $\epsilon$ in the conservation of the geometry, we take lower value of $\epsilon = 0.01$, and keep the remaining parameters equal. We solve the Hamilton-Jacobi equation to order $2$ and then simulate a trajectory starting from $(1,3,3)$. The evolution of the Hamiltonian and Casimir difference is plotted in Figure~\ref{fig:lv_hamiltonia_poisson_epsilon_001}.% $h = 2$, Hamilton-Jacobi was solved to order $2$, $alpha_\epsilon$ approximated to order $8$
\begin{figure}[H]
    \centering
    \includegraphics[width = 0.7\textwidth]{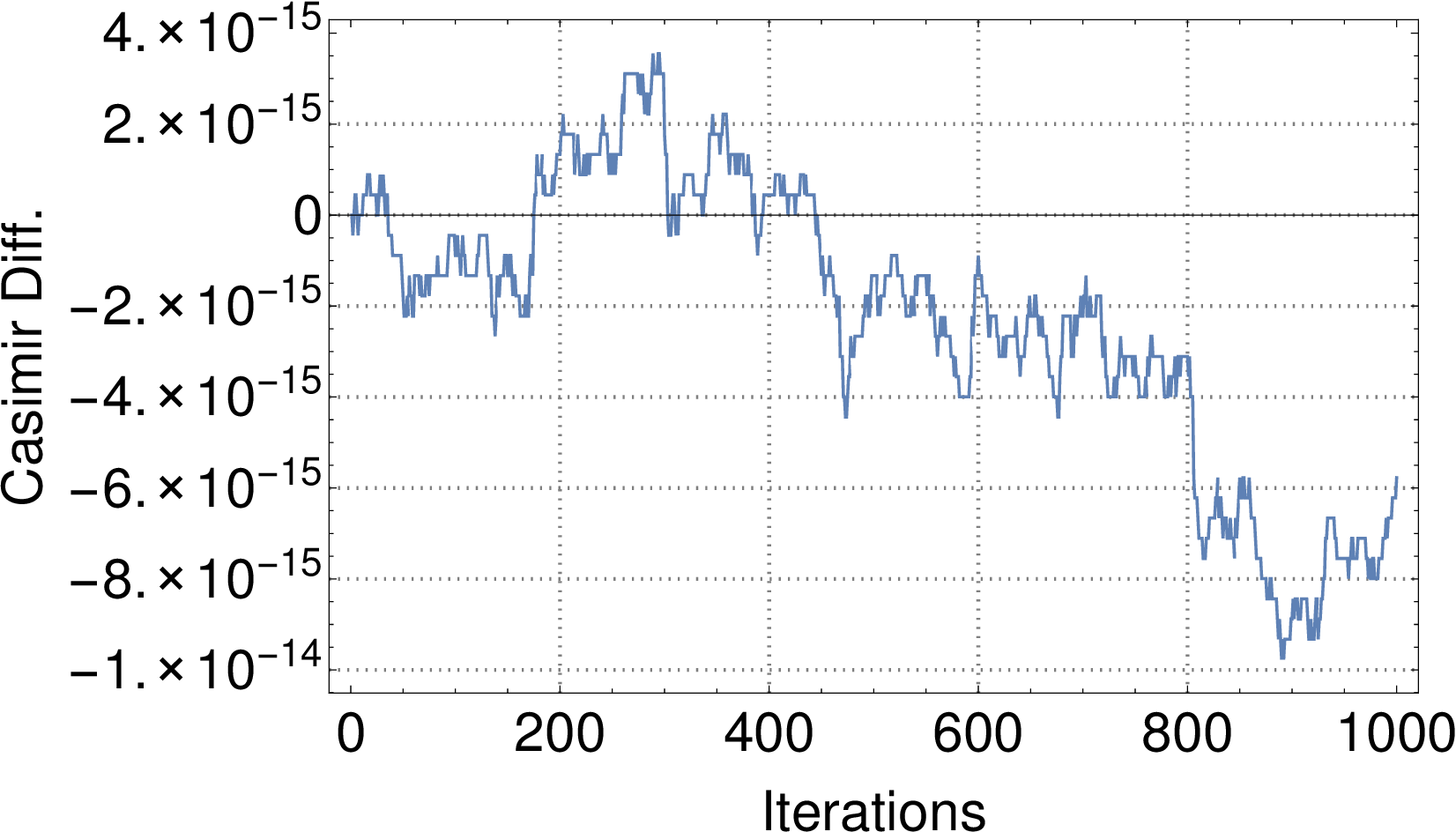}
    \includegraphics[width = 0.7\textwidth]{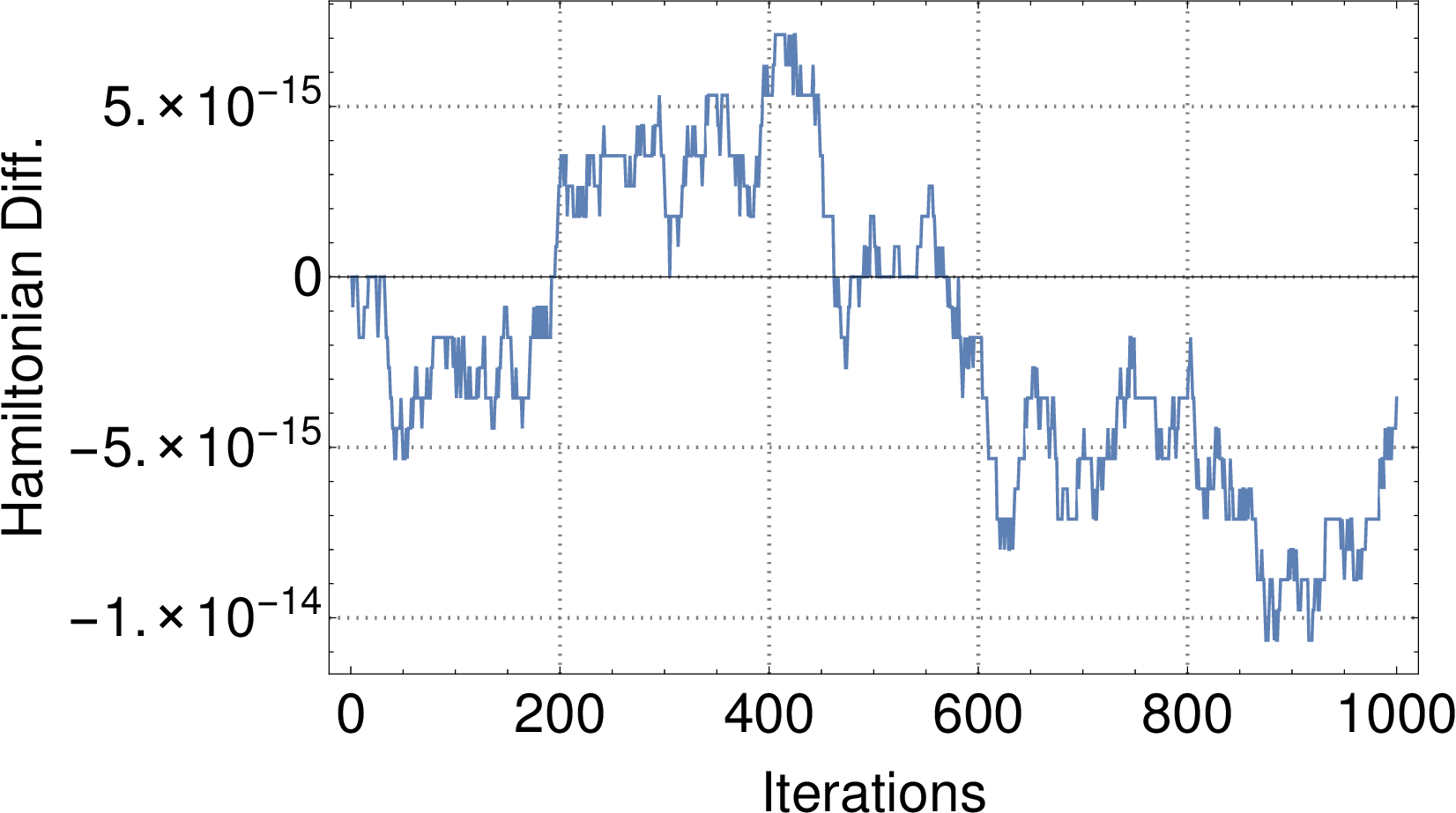}
    \caption{{\it (Top)}  Conservation of the Casimir. We observe a notable increase in accuracy in conservation of the Casimirs, due to the reduction in the value of $\epsilon$.  {\it (Bottom)} Evolution of the Hamiltonian difference through the simulated trajectory. We observe a high conservation of the Hamiltonian, similar to the previous case ($\epsilon = 0.1$).}
    \label{fig:lv_hamiltonia_poisson_epsilon_001}
\end{figure}

%
% Taken from the Mathematica code: formal_lotkaVolterra_example_1_casimir_diff_hamiltonian_diff_2.
%
Finally, we study the conservation of the Poisson tensor, studying the error described in eq.~\eqref{error}. The Figure~\ref{fig:lv_2} below shows nice conservation of the Poisson tensor, even surpassing our theoretical predictions.
\begin{figure}[H]
    \centering
    \includegraphics[width = 0.6\textwidth]{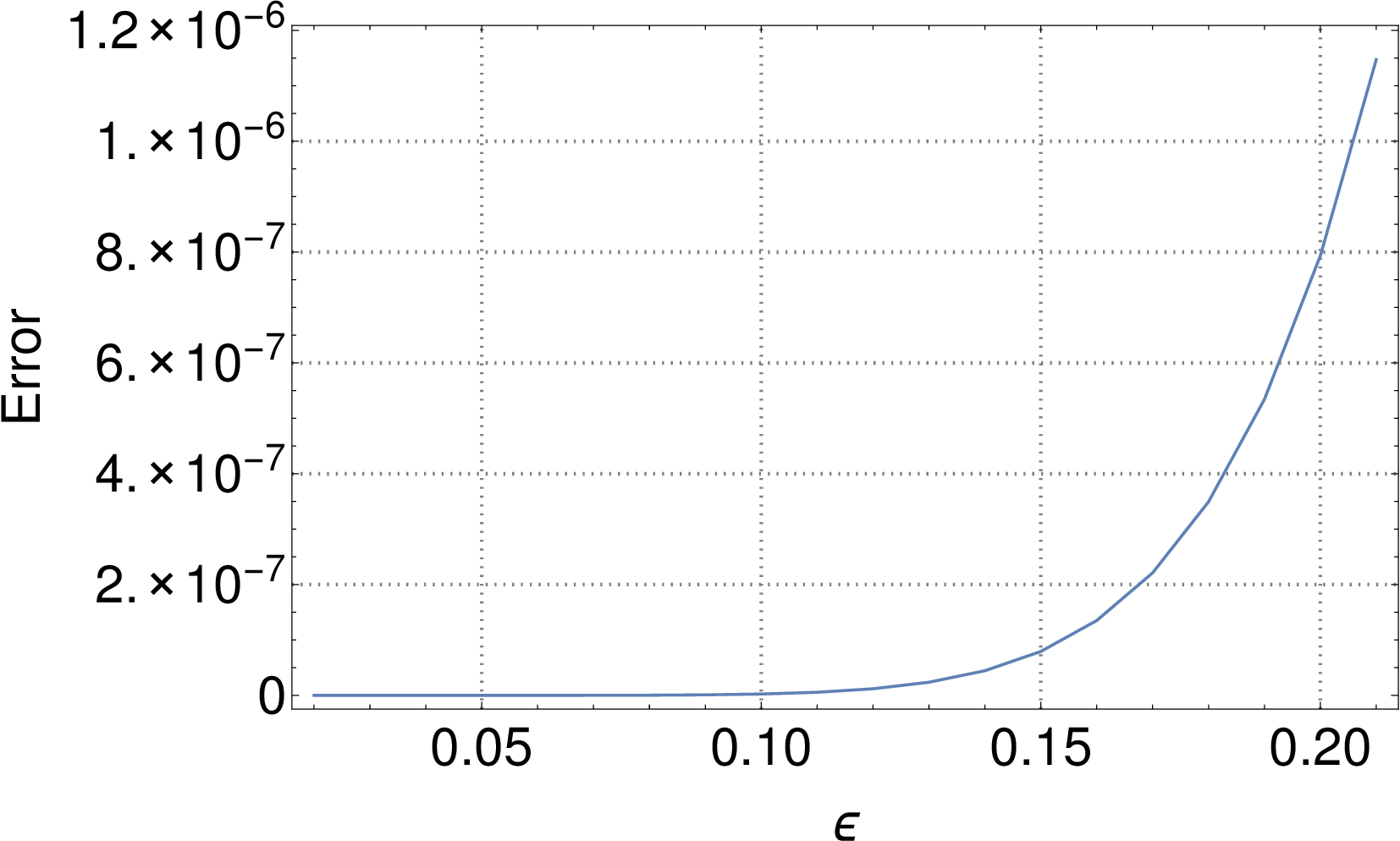}
    \includegraphics[width = 0.6\textwidth]{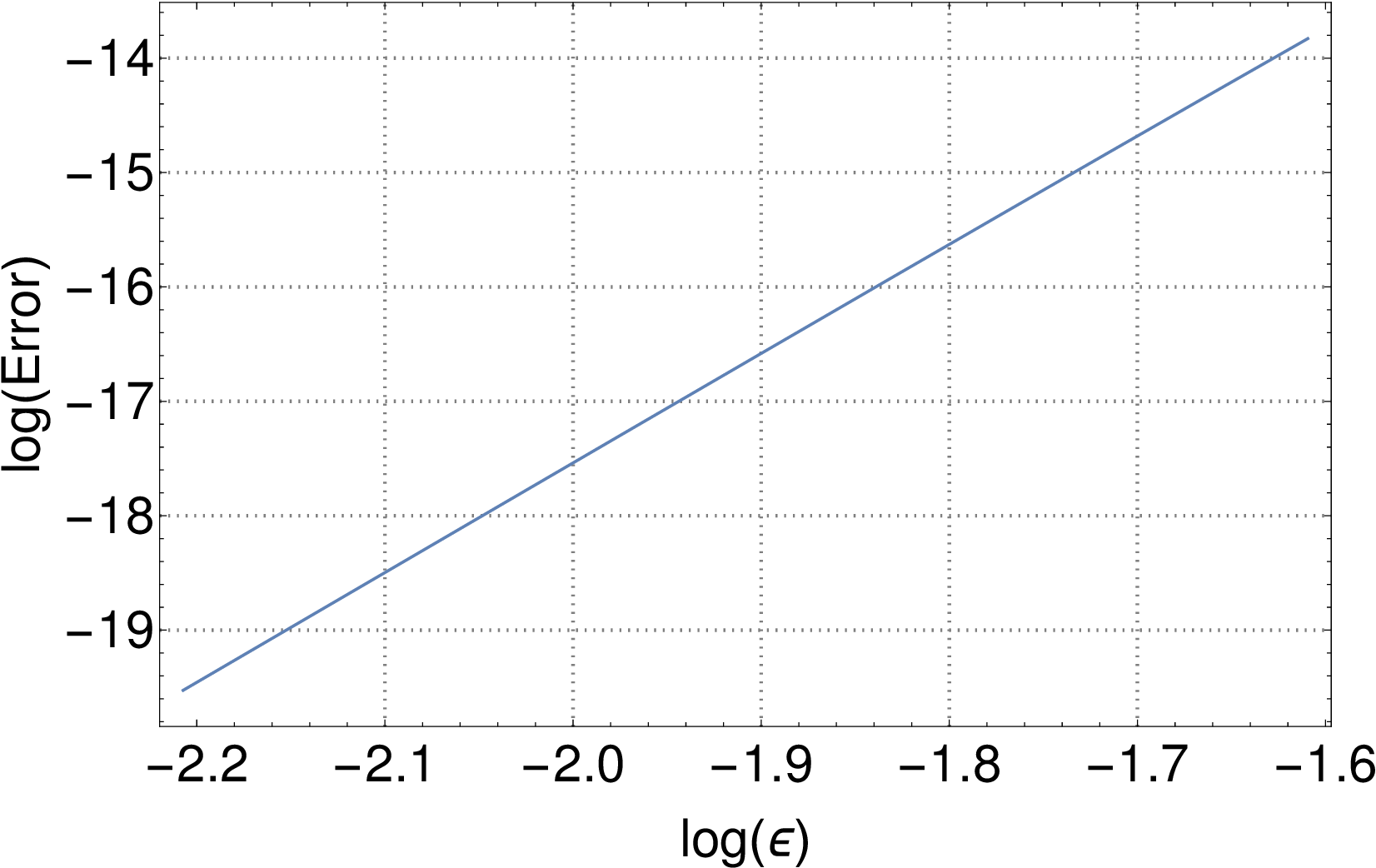}
    \caption{({\it Top:}) Conservation of the error as described in eq.~\eqref{error}, using the same setting as in Figure~\ref{fig:lv_hamiltonia_poisson_epsilon_001}. ({\it Bottom}) Plot equivalent to the one above, but representing the logarithm of $\epsilon$ versus the logarithm of the error. We observe a slope slightly larger than $9$, which matches our theoretical findings.}
    \label{fig:lv_2}
\end{figure}

%
% Taken from the Mathematica code: formal_lotkaVolterra_last.nb.
%
% \subsubsection{Implementation of Collective Integrators}

% Although in this paper we focus on the Hamilton-Jacobi approach, we provided a sound theoretical basis for the construction of collective integrators. Nonetheless, we present here a brief illustration of a possible implementation in the case of Lotka-Volterra system.

%%%%%%%%%%%%%%%%%%%%%%%%%%%%%%%%%%%%%%%%%%%%%%%%%%%%%%
\subsection{Example $3$: Non-canonical symplectic structure}
%%%%%%%%%%%%%%%%%%%%%%%%%%%%%%%%%%%%%%%%%%%%%%%%%%%%%%
The development of symplectic integrators is well-known when the symplectic form is in canonical form. Nonetheless, when the symplectic structure presents a different form there are no general methods to construct symplectic integrators. A well-known and interesting example is the case where the symplectic form is ``twisted'' by a magnetic term. See~\cite{marsden3} and the references therein for a thorough description. Aligned with that case, we consider the Poisson tensor with associated matrix
\[
\pi(x,y,z,p_x,p_y,p_z) = 
\left(
\begin{array}{cccccc}
 0 & 0 & 0 & -1 & 0 & 0 \\
 0 & 0 & 0 & 0 & -1 & 0 \\
 0 & 0 & 0 & 0 & 0 & -1 \\
 1 & 0 & 0 & 0 & x^2 & z \\
 0 & 1 & 0 & -x^2 & 0 & y \\
 0 & 0 & 1 & -z & -y & 0 \\
\end{array}
\right),
\]
which is a simple modification of the canonical symplectic form by the form $B = x^2dx\wedge dy + y dy\wedge dz + z dx \wedge dz$. Our approach hinges on the fact that thinking of the symplectic tensor as a Poisson structure brings the possibility of using the results of this paper. A detailed study of the construction of symplectic integrators in this settings will be carried elsewhere. However, we run here some preliminary test to assess the obtainment of symplectic realizations in this case through the techniques of this paper. The outcome is contained in Figure~\ref{fig:non_canonical_hamiltonian}. The accuracy of this realization  is very important, as the design of integrators depend on it.

% For simplicity, we consider a quadratic Hamiltonian of the form
% \[
% H(x^1,x^2,x^3,p^1,p^2,p^3)= x^2 + y^2 + px^2.
% \]
% We take $\epsilon = 0.1$, stepsize $h = 2$, and solve Hamilton-Jacobi  to order $3$. The mapping $\alpha_\epsilon$ and $\beta_\epsilon$ are approximated to order $4$. In the Figure~\ref{fig:non_canonical_hamiltonian} below we illustrate the evolution of the Hamiltonian for $1000$ iterations.
\begin{figure}[H]
    \centering
    \includegraphics[width = 0.6\textwidth]{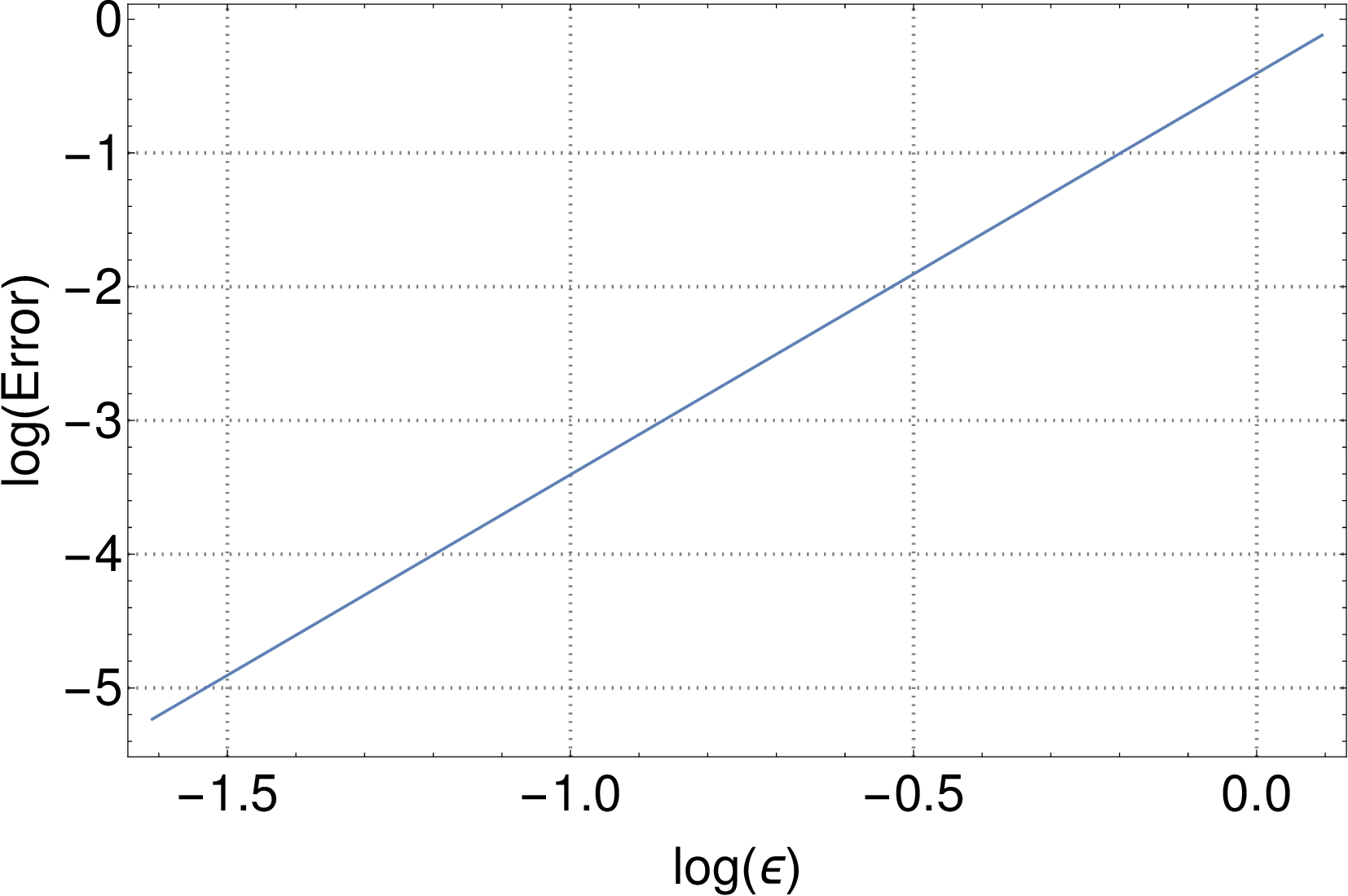}
    \caption{Plot of the logarithm of $\epsilon$ versus the logarithm of the error of the symplectic realization. Error here is measured like in eq.~\eqref{sym:error}. We approximate $\alpha_\epsilon$ to order $2$, so we expect the approximation to be $\mathcal{O}(\epsilon^3)$. This matches exactly the plot, which happens to have a slope of $3$.}
    \label{fig:non_canonical_hamiltonian}
\end{figure}

This paves the way for the obtainment of integrators through the Hamilton-Jacobi theory, collective integrators or any other means.
%
% Taken from the Mathematica code: formal_magneticTerm
%

%%%%%%%%%%%%%%%%%%%%%%%%%%%%%%%%%%%%%%%%%%%%%%%%%%%%%%%%%%%%%%%%5
\section{Conclusions and Future Work}\label{sec:conclusions}

In this paper we advocate  a new strategy for the design of geometric Poisson integrators, where geometry and dynamics are approximated through different processes. The introduced philosophy is promising, both from the theoretical and numerical viewpoints, and paves the way for new methodologies to be exploited in forthcoming studies. We would like to highlight here some of the main points and future thoughts.

\begin{itemize}
    \item {\it Other approximation techniques} should be studied to push the introduced methods further. Our methods rely on the simplest approximation of ODE, just Taylor series approximation. The presented method works well for simple systems, but more elaborated techniques should be studied to outperform and broaden the results introduced here.

    \item For high order approximations the designs presented here may have a {\it high computational burden which slows the computation}. Also, the geometric approximation could become hard at points where the momenta $p$ are large, as the ODE for the approximation of $\alpha_\epsilon$ has a larger error. This raises the questions of which other approximations may be used to alleviate these issues.

    \item We envision {\it applications to learning Poisson dynamical systems} as a by-product of the results of this paper. In a very recent paper (\cite{vaquero2023symmetry}) techniques relying on symplectic groupoids have been used to learn and simulate dynamical systems with symmetry, but the the results where limited to scenarios where the symplectic groupoid was explicitly known before hand. 

\item {\it Backward error analysis} seems not to fall directly into our setting. The fact that the geometry is only conserved approximately seems to be the main obstruction. Nonetheless, a full understanding of the underlying geometry would uncover intriguing properties, like the symmetries explored in Prop.~\ref{prop:Sodd}. Also, by applying a version of backward error analysis, one should demonstrate the near-preservation of the Hamiltonian.

% \marginAC{AC, more questions:

% + backward error analysis? find $\hat \pi_\epsilon, \hat H_\epsilon$ generating $\hat \phi_{\epsilon,t}$

% + preservation of symplectic forms on leaves?
% }
% \marginDM{I think that is a consequence of backwards error analysis for the symplectic flow on the local symplectic groupoid, then the system upwards almost preserves the energy.}
% \marginMV{I do not remember the original question. But if Casimirs are (almosts) conserved, then symplecitc leaves are also almost conserved, no?}
% \marginAC{only for regular leaves... I think one should give a more geometric proof of preservation of leaves (here or elsewhere)}

\item {\it Exploiting the underlying local symplectic groupoid multiplication} is still yet to be fully explored regarding their applications in geometric integration methods. In particular, the generating functions for the groupoid structure described in \cite{ca22} can be put into play in an analogous context of the methods of this paper. These topics will be studied in \cite{cacanar}.
    
\end{itemize}

\section*{Acknowledgements}
AC acknowledges the support of the Brazilian agency grants CNPq 309847/2021-4, 402320/2023-9 and FAPERJ JCNE E-26/203.262/2017. AC also thanks ICMAT, Madrid, for the hospitality and support during parts of the elaboration of this project.
DMdD and MV acknowledge financial support from the Spanish Ministry of Science and Innovation under grants PID2022-137909NB-C21, RED2022-134301-TD and DMdD from the Severo Ochoa Programme for Centres of Excellence in R\&D (CEX2023-001347-S).

\bibliography{references}
\bibliographystyle{acm}

\end{document}